\documentclass[reqno,12pt,a4paper]{amsart}
\usepackage{pgfplots}
\usepackage[english]{babel}
 \pdfoutput=1
\usepackage{amsmath,amsthm,amssymb,amsfonts, mathdots}
\usepackage{bbm}
\usepackage{scrtime, cancel}
\usepackage{enumitem, color, comment}

\definecolor{pinegreen}{rgb}{0.0, 0.47, 0.44}
\definecolor{brilliantrose}{rgb}{1.0, 0.33, 0.64}
\usepackage{mathtools}
\usepackage{ifpdf}
\ifpdf
\usepackage[backref]{hyperref}
\else
\usepackage[hypertex]{hyperref}
\fi
\usepackage{pdfsync,verbatim}
\usepackage{color}
\mathtoolsset{showonlyrefs}

\numberwithin{equation}{section}   %%numera le equazioni sezione per sezione

\allowdisplaybreaks
\newtheorem{theorem}{Theorem}[section]
\newtheorem{lemma}[theorem]{Lemma}
\newtheorem{proposition}[theorem]{Proposition}
\newtheorem{corollary}[theorem]{Corollary}

\theoremstyle{definition}
\newtheorem{remark}[theorem]{Remark}

\usepackage{graphicx}

\makeatletter
\newcommand*\bigcdot{\mathpalette\bigcdot@{.5}}
\newcommand*\bigcdot@[2]{\mathbin{\vcenter{\hbox{\scalebox{#2}{$\m@th#1\bullet$}}}}}
\makeatother

\newcommand{\R}{\mathbb{R}}

\newcommand{\N}{\mathbb{N}}

\DeclareMathOperator{\tr}{tr}             % trace
\newcount\dotcnt\newdimen\deltay
\def\Ddot#1#2(#3,#4,#5,#6){\deltay=#6\setbox1=\hbox to0pt{\smash{\dotcnt=1
\kern#3\loop\raise\dotcnt\deltay\hbox to0pt{\hss#2}\kern#5\ifnum\dotcnt<#1
\advance\dotcnt 1\repeat}\hss}\setbox2=\vtop{\box1}\ht2=#4\box2}

\newcommand{\norm}[1]{\left\lVert#1\right\rVert}

\newcommand{\Red}{\color{red}}

\pgfplotsset{compat = newest}

\makeatletter
\def\author@andify{%
  \nxandlist {\unskip ,\penalty-1 \space\ignorespaces}%
    {\unskip {} \@@and~}%
    {\unskip \penalty-2 \space \@@and~}%
}
\makeatother

\title[Sharp variational, jump and oscillation bounds]{%%Strong and weak type  
Sharp variational, jump and oscillation bounds \\ in a general Gaussian context}

\author{Valentina Casarino}
\address{DTG, Universit\`a degli Studi di Padova\\ Stradella san Nicola 3 \\I-36100 Vicenza \\ Italy}
\email{valentina.casarino@unipd.it}
\author{Paolo Ciatti}
\address{Dipartimento di Matematica "Tullio Levi Civita", Universit\`a degli Studi di Padova\\Via Trieste, 63, 35131 Padova,  \\ Italy}
\email{paolo.ciatti@unipd.it}
\author{Peter Sj\"ogren}
\address{Mathematical Sciences,  University of Gothenburg and  Mathematical Sciences,
Chalmers University of Technology  \\ SE - 412 96 G\"oteborg, Sweden}
\email{peters@chalmers.se}



\thanks{The first and second authors are members of the Gruppo Nazionale per l'Analisi Matematica, la Probabilità e le loro Applicazioni (GNAMPA)
of the Istituto Nazionale di Alta Matematica (INdAM).
This research was carried out while the third author was GNAMPA Professore Visitatore
at the University of Padova, Italy. The third author also profited from a grant from the Royal Society
of Arts and Sciences in Gothenburg and another from the Royal Swedish Academy of Sciences. He
is also grateful to the University of Padova for its hospitality during several visits.
}

\date{\today}

\begin{document}

\maketitle
 
\begin{abstract}
 We consider a general, nonsymmetric Ornstein--Uhlenbeck semigroup $(\mathcal H_t)_{t>0}$.
 We prove an  $L^p$ bound for 
the jump quasi-seminorms  for $1 < p < \infty$ and 
a weak type (1,1) oscillation inequality,  both with respect to the invariant measure.
These results are established for the
 order  
$\varrho=2$.
To do so, 
we analyze specific components
 of $(\mathcal H_t)_{t>0}$, by distinguishing between small and large values of $t$, and between local and global spatial zones.
This  decomposition allows us to explicitly identify which parts of the semigroup remain bounded 
and which  are responsible for the failure of boundedness, both in a weak and in a strong sense, 
even with respect to Lebesgue measure.

\end{abstract}

 \section{Introduction}
This paper is the last in a series \cite{CCS6, CCS8, CCS9}
where we started  studying variational estimates for the Ornstein--Uhlenbeck  semigroup. 
Here we conclude our analysis, arriving at a complete picture of  $\varrho$-th order variational, jump, and
oscillation  inequalities for
$\varrho \in [1,+\infty)$. The formal definitions 
of variation, jump, and oscillation seminorms, 
along with the main statements of this paper, will be given in Section \ref{S.var}.

The study of variational inequalities has recently become a central theme in harmonic analysis and ergodic theory, 
providing a refined understanding of the convergence of families of operators beyond the classical maximal estimates. 
Inspired by the pioneering work of J. Bourgain, V. F. Gaposhkin  and D. Lépingle  
and the subsequent developments by  M. Mirek, W. S\l\,\!omian,  E. M. Stein,  T. S. Szarek, J. Wright,  B. Wróbel and P. Zorin-Kranich, 
\cite{Lepingle, Gaposhkin1, Gaposhkin2, Bourgain, Bourgain_etal, Jones2, Mirek1, Mirek2, Mirek3, Mirek4, Mirek5, Mirek6}, 
our main objective is 
to provide a complete characterization of the $L^p$ and weak type $(1,1)$ boundedness for $\varrho$-variations, oscillations, 
and $\lambda$-jumps associated with a general, not necessarily symmetric, 
Ornstein-Uhlenbeck semigroup $(\mathcal{H}_t)_{t>0}$.

The Ornstein-Uhlenbeck semigroup on $L^p(\mathbb{R}^n)$ is given by the classical Kolmogorov representation  %\cite{Kolmo}
\begin{equation}\label{Kolmo}
\mathcal H_t
f(x)=
\int  
f(e^{tB}x-y)d\gamma_t (y)\,, \quad x\in\R^n\,,
\end{equation}
where $B$ is a real matrix whose eigenvalues have negative real parts, and $d\gamma_t$ is a normalized Gaussian measure, given by
\[
d\gamma_t (x)=
(2\pi)^{-\frac{n}{2}}
(\det \, Q_t)^{-\frac{1}{2} }
e^{-\frac12 \langle Q_t^{-1}x,x\rangle}
dx  \,,\qquad \text{ $t\in (0,+\infty]$.}\]
Here the  $Q_t$ are  covariance  matrices, defined as
\begin{equation}\label{defQt}
Q_t=\int_0^t e^{sB}Qe^{sB^*}ds, \qquad \text{ $t\in (0,+\infty]$},
\end{equation}
where  $Q$ is  a real, symmetric and positive definite $n\times n$ matrix.

The measure $d\gamma_\infty$ defined above is the unique probability measure which is invariant
with respect to the semigroup. Being a Gaussian measure, it is not globally doubling.
Each $\mathcal{H}_t$
 is an integral operator with respect to $d\gamma_\infty$ defined by the Mehler kernel.

Our investigation is motivated by the fact that for symmetric semigroups the critical threshold for the boundedness of
 variational, jump, and
oscillation  operators is  $\varrho=2$. 
We shall see that this threshold remains in our more general case. Only for some parts of
the semigroup it is possible to obtain variational estimates  for $1\le \varrho\le 2$.

In the symmetric case, where the operators $\mathcal{H}_t$ are self-adjoint on $L^2(\gamma_\infty)$, 
the $L^p$ boundedness of the variation for $\varrho > 2$ 
and of the oscillation for $\varrho = 2$ follows from \cite{JR} and 
is based on  Rota's Alternierendes Verfahren theorem 
\cite{Rota}. 
Regarding jumps,   when $\varrho=2$ the $L^p$ boundedness follows from a theorem holding for a large class of symmetric semigroups 
\cite[Theorem 1.2]{Mirek1}. 
We refer to \cite{Harboure, Crescimbeni, Betancor1,  Ma1, Ma2, Betancor2, Betancor3}
for several other variational estimates  in a symmetric Gaussian context.

However, in the nonsymmetric Gaussian setting the lack of symmetry  presents significant technical challenges, and only a few results were known \cite{AlmeidaMedit, Almeida,
 Le Merdy}. 
To treat this case, we exploit the asymptotic behavior of the Mehler kernel, separating  small and large values of $t$.
We further distinguish between
   a local region, where the Gaussian measure behaves like the Lebesgue measure,
  and a global region, where the Gaussian density dictates the decay of the operators. 

From the variational point of view, the  worst case occurs  when $(x,u)$ belong 
to  the local region and $t$ is small. (see \cite[Section 8]{CCS8}).

\medskip

While it is known that no variational bounds 
hold at $\varrho = 2$
 for the full Ornstein--Uhlenbeck semigroup \cite{Qian, CCS8}, 
 both jump and oscillation inequalities have remained largely unexplored at or below this endpoint.
The main contributions of this work are twofold. 
First of all, we establish in Theorem \ref{thm:main1} the $L^p(\gamma_\infty)$ boundedness 
of the jump quasi-seminorms of order $\varrho=2$ for $1 < p < \infty$ 
(we mention in passing that the weak type (1,1) was recently proved for $\varrho=2$  in \cite{CCS9}).
Then, in Theorem \ref{thm:main2}, we prove a weak type (1,1) oscillation inequality with respect to the invariant measure.

Moreover, 
we show that these  bounds are sharp, demonstrating the failure of such estimates for $\varrho < 2$.
Our counterexample is based on Rademacher functions and is inspired by  a celebrated, one-dimensional  counterexample due to Qian \cite{Qian}.

Finally, we discuss the change from Gaussian to Lebesgue measure $dx$,
 illustrating how  these bounds fail  in the $L^p(dx)$ setting.

\medskip

However, as already observed in \cite{CCS8}, one can obtain 
better bounds by analyzing specific components of the Mehler kernel 
rather than the full kernel. 
This approach allows us to establish variational, oscillation, and jump inequalities 
for any $\varrho\ge 1$  for certain parts of the Ornstein--Uhlenbeck semigroup, 
whereas such bounds are typically only known for $\varrho\ge 2$. 
Table 1 summarizes these results.

 \medskip
\begin{table}[htbp]
    \centering
    \begin{tabular}{l | c c c c}
        % Intestazioni delle colonne (spazio in alto)
        & { Local region, } & { Local region, } &{ Global region, } \\
        &  \textbf{$t\in(0,1]$}  &   {$t\ge 1$}  & {$0<t<\infty$}    \\\hline
        % Riga 1 con intestazione a sinistra
        ${\varrho-\text{var.},\, \varrho>2}$  & B & B &B \\
        % Riga 2 con intestazione a sinistra
        ${\varrho-\text{var.},\, {{\varrho\in [1,2]}}}$ & {\Red{U}} &B& B \\
        % Riga 3 con intestazione a sinistra
         ${\varrho-\text{oscill.},\, \varrho\ge 2}$&  B & B& B  \\
        % Riga 4 con intestazione a sinistra
         ${\varrho-\text{oscill.},\,{{ \varrho\in[1,2)}}}$ & {\Red{U}} &  B& B  \\
        % Riga 5 con intestazione a sinistra
         ${\varrho-\text{jumps},\, \varrho\ge 2}$  & B & B & B \\
        % Riga 6 con intestazione a sinistra
         ${\varrho-\text{jumps},\, {{\varrho\in [1,2)}}}$  & {\Red{U}} &   B&  B  \\
    \end{tabular}
    \vspace{0.2cm} % Spazio tra la tabella e la didascalia
    \caption{$L^p$ boundedness and weak type (1,1)  for the Ornstein--Uhlenbeck semigroup, with respect to the invariant measure. }
    \label{tab:mia_tabella}
\end{table}

 To the best of our knowledge, 
 this is the first instance of a nonsymmetric semigroup 
 where variation, jumps, and oscillation
  of order $\rho$ are completely investigated for all $\rho \ge 1$.

\subsection*{Contents of the sections } 
Section \ref{S.var} introduces variational, jump and oscillation inequalities
and outlines the state of the art for the Ornstein--Uhlenbeck semigroup. We also state our
two main results, Theorems  \ref{thm:main1} and  \ref{thm:main2}, which 
   establish a strong-type jump inequality and a weak-type $(1,1)$ oscillation inequality, respectively, both at $\varrho = 2$.  
 In Section \ref{s:prelim}
 explicit formulas for the Mehler kernel and  some upper bounds 
 for its time derivative are provided. 
 We also describe  the spatial localization procedure, 
 and  recall the hypercontractivity properties   of $\mathcal H_t$. These are used in 
  Section \ref{s:strong_type_large_times} to show that  the variation operator of order $\varrho\ge 1$
 for large times is bounded on $L^p(\gamma_\infty)$
 This implies  $L^p$ boundedness for the  jump operator  and for the oscillation
 as well, in the same range of $\varrho$ and $t$.
 
In Section \ref{s:smallt_global}
we finish the proof of Theorem \ref{thm:main1} 
by proving the $L^p(\gamma_\infty)$
 boundedness of the jump operator for small times in the global region. 
The argument relies on  dyadic decompositions of the time parameter, which allows us to estimate the operator piece by piece.
In Section \ref{s:oll} we complete
 the proof of  Theorem \ref{thm:main2} by establishing 
uniform  oscillation bounds for the semigroup.
The analysis is reduced to the local component for small times, 
which is then decomposed into difference operators and a main operator.
By applying known bounds and an earlier result by H. Liu \cite{Liu} 
to the main part, we obtain weak type (1,1) 
and strong type $(p,p)$
 estimates for the oscillation. 
Section \ref{s:Qian} shows the failure 
of uniform oscillation and jump bounds for 
$\varrho<2$.

Finally, in Section \ref{s:Leb}
we replace $d\gamma_\infty$ 
by Lebesgue measure.

\subsection*{Notation}

 We shall denote by $C<\infty$ and $c>0$ constants that may vary from place to place. 
 Unless otherwise explicitly indicated, these constants depend only on $n$, $Q$, $B$, and
on $p$ if relevant.
For two non-negative quantities $a$ and $b$, we write $a\lesssim b$, or equivalently  $b\gtrsim a$, if $a\leq C b$ for some $C$.
The symbol $a\simeq b$ 
means that $a\lesssim b$ and $b\lesssim a$.
By $\mathbb N$ we mean $\{0,1,\dots\}$, and $\N_+$ will denote the set of strictly postive integers.
If $A$ is an $n \times n$ matrix, we write $\|A\|$ 
for its operator norm on $\R^n$ endowed  with the Euclidean norm.
The symbol $\mathcal I$ will always denote an interval in $\R_+$.
We shall adopt the  dot notation for differentiation with respect to the time variable $t$,
  writing  $ \dot K_t = \partial_t K_t$.

\medskip

The authors would like to thank Professor Mariusz Mirek for several enlightening discussions on the subject of this paper.
The first author would like to thank Professor A. Rhandi, 
 for bringing to her attention the case discussed in Section \ref{s:Leb}.

\section{Variational, jump and oscillation inequalities}\label{S.var}

 In this section, we define these three types of inequalities for the Ornstein-Uhlenbeck semigroup  $\left(\mathcal H_t \right)_{t>0} $, and describe what is known at present.  We make here no distinction between small and large times or local and global regions.  
 For  variational inequalities in more general contexts, we refer the reader to   the recent monograph \cite{Krause}.

\subsection{Variation seminorms}\label{subs:var} ~
If 
$\mathcal I {\color{black}{\subset}} \R_+$ is an interval and $f \in L^1(\gamma_\infty)$,   the $\varrho$-th order variational seminorm of $\mathcal H_t f(x)$   on   $\mathcal I $ is defined by
           \begin{equation}\label{def:intro_rho2_OU}
  \|\mathcal H_t \,f(x)\|_{v(\varrho), \mathcal I} = \sup\, \left( \sum_{i=1}^N |\mathcal H_{t_i} f(x)- \mathcal H_{t_{i-1}}f(x)|^\varrho \right)^{1/\varrho},\qquad x\in \R^n,\quad
  1 \le \varrho < \infty,
\end{equation}
where 
 the supremum is taken over all finite, increasing
 sequences $\left(t_i \right)_0^N,  \; N \ge 1$,   of points in $\mathcal I$.
 
Regarding variational inequalities for a general Ornstein--Uhlenbeck semigroup,  the picture is complete, since the following facts  are known:
\begin{enumerate}
 \item
 For $\varrho>2$  the variation operator mapping $f\in L^p(\gamma_\infty)$ to the function 
\begin{equation}\label{def:fnct}
x\mapsto \|\mathcal H_tf(x)\|_{v(\varrho),\R_+} 
\end{equation}
is bounded 
from $L^p(\gamma_\infty)$ to $L^{p}(\gamma_\infty)$, $1<p<\infty$,   (see \cite[Corollary 4.5]{Le Merdy} and   \cite{Almeida}). 
 \item
For $\varrho>2$  the same operator 
is  of weak type $(1,1)$ for $d\gamma_\infty$, that is,
\begin{equation}\label{ineq:var_weak}
\|\, \|\mathcal H_t \,f(x) \|_{v(\varrho), \R_+}\|_{L^{1,\infty}(\gamma_\infty)}\lesssim_{\varrho}  \|f\|_{L^1(\gamma_\infty)},
\end{equation}
(see \cite{CCS8}, and \cite{CCS6} for a different proof in the one-dimensional case).
\item 
Items (1) and (2)  are false  for  $1\le \varrho\le 2$, since  the  operator
$    f \mapsto \|\mathcal{H}_t f (x)\|_{v(2),\R_+}$, $x\in \R^n$,
  is not of strong nor weak type $(p,p)$ with respect to $d\gamma_\infty$,  for any $p \in [1,\infty)$
(see \cite[Section 8]{CCS8}).
\end{enumerate}
\medskip

\subsection{Jump quasi-seminorms}\label{subs:jump}~
Given an interval  $\mathcal I \subset \mathbb R_+$,  $f \in L^1(\gamma_\infty)$ and $\lambda>0$, the $\lambda$-jump counting function for 
$\mathcal H_t\, f$ in $\mathcal I$ at   $x\in\R^n$
is defined as
\begin{multline}\label{def:jump}
 N_\lambda (\mathcal H_t f(x):\:t\in\mathcal I)=     
 \sup\big\{
N \in\mathbb N_+:\,
\text{  there exist points }\,
  t_0 < t_1 < \dots < t_N \text{ in $\mathcal I$ }   
 \\
\text{ such that }\,
|\mathcal H_{t_{i}} f(x) - \mathcal H_{t_{i-1}} f(x)| > \lambda
 \,\text{ for } \, i=1,\dots, N
 \big\}.
 \end{multline}
Then the jump quasi-seminorm  of order $\varrho$ of
$\mathcal H_t\, f$ in $\mathcal I$ is              
\begin{equation}\label{jump_seminorm_OU}
 J_\varrho^p  (\mathcal H_t f:\,t\in\mathcal I) = 
 \sup_{\lambda>0}
 \| \lambda \, \big(N_\lambda \big(\mathcal H_t f(\cdot):  t\in\mathcal I \big)\big)^{1/\varrho}\|_{L^p(\gamma_\infty)},
\quad 1 \le p,\varrho < \infty.
\end{equation}
Notice that the quasi-seminorms $ J_\varrho^p(\cdot) $ are decreasing  in $\varrho$.

The notion of 
weak jump quasi-seminorms  may be given in a similar way, 
by replacing  the Lebesgue space $L^p(\gamma_\infty)$   in  \eqref{jump_seminorm_OU}  by a   Lorentz space $L^{p,q}(\gamma_\infty)$
(see  \cite[formula (1.2)]{Mirek1});
we are mainly interested in  the case $(p,q)=(1,\infty)$, where 
\begin{equation}\label{jump_seminorm_weak}
 J_\varrho^{1,\infty}  (\mathcal H_t f:\,t\in\mathcal I)=
 \sup_{\lambda>0}
 \| \lambda \, \big(N_\lambda \big(\mathcal H_t f(\cdot):  t\in\mathcal I \big)\big)^{1/\varrho}\|_{L^{1,\infty}(\gamma_\infty)}.
\end{equation}

The main question studied concerning the jumps is whether $J_\varrho^p  (\mathcal H_t f:\,t\in\mathcal I)$  is controlled by $\|f\|_{L^p(\gamma_\infty)}$,
or  $J_\varrho^{1,\infty}  (\mathcal H_t f:\,t\in\mathcal I)$ by $\|f\|_{L^1(\gamma_\infty)}$.   

A central role in our work is played by the following estimate
\begin{equation}\label{domination}
 J_\varrho^p  (\mathcal H_t f:\,t\in\mathcal I) \le \| \,\|\mathcal H_t \,f(\cdot) \|_{v(\varrho), \mathcal I}\|_{L^p(\gamma_\infty)},\qquad \varrho \ge 1,\quad p\ge 1.
\end{equation}
stating that  the  jump quasi-seminorms
$ J_\varrho^p $ are dominated by 
the $L^p$ norms of 
variational seminorms of order $\varrho$. This                   
follows from  the inequality
\begin{align}\label{dom_pointwise}
\lambda\big( N_\lambda (\mathcal H_t f(x):\,\,t\in\mathcal I)\big)^{1/\varrho}\le   \|\mathcal H_t \,f(x)\|_{v(\varrho), \mathcal I}\,, \qquad  x\in\R^n,
\end{align}
holding for all $\varrho\ge 1$ and  $\lambda>0$, which is essentially Chebyshev's inequality in $\ell^\varrho$. 
                           
Regarding jump inequalities for a general Ornstein--Uhlenbeck semigroup, the following facts  are known:
\begin{enumerate}
 \item[(1')]
For $1<p<\infty$ and $\varrho>2$,
the  jump quasi-seminorms  of the  Ornstein--Uhlenbeck semigroup 
 $\left(\mathcal H_t\right)_{t>0}$ are bounded on $L^p(\gamma_\infty)$, that is,
\begin{equation*}
 J_\varrho^p  (\mathcal H_t f:\,t>0) \le  C\, \|f\|_{L^p(\gamma_\infty)};
\end{equation*}
 this is immediate from \eqref{domination} and item (1)).
 \item[(2')]
For   all $f\in L^1(\gamma_\infty)$ one has
\begin{equation}\label{ineq:weaktype} J_2^{1,\infty} (\mathcal H_t f:\,t>0) \le C\,\|f\|_{L^1(\gamma_\infty)};
 \end{equation}
 see \cite[Theorem  1.2]{CCS9}.
\item[(3')]
In the light of item (3),
the weak  jump quasi-seminorms $ J_\varrho^{p,\infty} (\mathcal H_t \,f :\,t>0)$ are not bounded by $\|f\|_{L^p(\gamma_\infty)}$ or $\|f\|_{L^{p,\infty}(\gamma_\infty)}$   for any
$\varrho\in[1,2)$, $1\le p<\infty$. In fact, 
  it is known that
for all $1\le\varrho<\varrho'< \infty$ 
\begin{equation}\label{ineq:Mariusz}
\|\, \|\mathcal H_t \,f(x) \|_{v(\varrho'), \R_+}\|_{L^{p,\infty}(\gamma_\infty)}
 \lesssim_{\varrho, \,\varrho'}  J_\varrho^{p,\infty} (\mathcal H_t \,f:\,t>0).
\end{equation}
For this we refer to \cite[formula (1.13)]{Mirek5} or to  \cite[formula (2.20)]{Mirek4}, both  based on an earlier remark  by Bourgain \cite{Bourgain}.

\item[(4')]
As a consequence of (3'),
the   jump quasi-seminorms $ J_\varrho^{p} (\mathcal H_t \,f :\,t>0)$ are not bounded by $\|f\|_{L^p(\gamma_\infty)} $
 for any
$\varrho\in[1,2)$ and $1<p<\infty$.
\end{enumerate}

\begin{remark}
For $\varrho>2$ and any interval $\mathcal I\subset \R_+$ one has
\begin{equation}\label{domination3}
\|\,\|\mathcal H_t \,f\|_{v(\varrho), \mathcal I}\|_{L^{1,\infty}(\gamma_\infty)}
\lesssim_{\varrho}
 J_2^1 (\mathcal H_t \,f:\,t\in\mathcal I) \le \| \,\|\mathcal H_t \,f \|_{v(2), \mathcal I}\|_{L^1(\gamma_\infty)},
\end{equation}
with the first inequality (which is a consequence of \eqref{ineq:Mariusz})
 failing for $\varrho=2$. 
The second bound is \eqref{domination} with $\varrho = 2$ and $p=1$.
In  view of \eqref{domination3}, 
the jump inequality \eqref{ineq:weaktype} of weak type $(1,1)$      
may be seen as  an endpoint refinement at $\varrho=2$     
of the variational inequality \eqref{ineq:var_weak} described in item (2).
 
\end{remark}

However,  
 nothing is known about the $L^p$ boundedness of jumps for $\varrho=2$.
The first result in this paper is exactly a  strong type jump inequality 
for $ \mathcal H_t$ when $\varrho=2$. 
\begin{theorem}\label{thm:main1}
 For $1<p<\infty$,
the  jump quasi-seminorms of order $\varrho =2$  for the  Ornstein--Uhlenbeck semigroup 
 $\left(\mathcal H_t\right)_{t>0}$ are bounded on $L^p(\gamma_\infty)$, that is,
\begin{equation*}
 J_2^p  (\mathcal H_t  f:\,t>0)\le  C\, \| f\|_{L^p(\gamma_\infty)}.
\end{equation*}
\end{theorem}

It should be mentioned that in the symmetric case (that is, when each  $\mathcal H_t$ is a self-adjoint operator)
 the boundedness of $J^p_2(\mathcal H_t f:t>0)$  follows
from
\cite[Theorem 1.2]{Mirek1}.

\subsection{Oscillation inequalities}~
Given a finite  increasing sequence $(t_{i})_{0}^N$ in $\mathcal I\subset \mathbb{R}^{+}$  and $x\in\R^n$,
the oscillation of order $\varrho$, $1\le\varrho<\infty$,
of $\mathcal{H}_t f(x)$  in $\mathcal I$ is defined as
\[  \mathcal O^\varrho_{ \{t_i\}, N}   (\mathcal H_t f(x) :\,t\in \mathcal I)
=\left(\sum_{i=1}^{N}\sup_{t_{i-1}\le
t\le t_{i}}| \mathcal H_t \,f(x)-
 \mathcal H_{t_i} \,f(x)|^{\varrho}\right)^{1/\varrho}.\]
    This defines a   seminorm.
    
With $1\le p<\infty$, a uniform oscillation 
inequality is a bound of the form 
\[\sup_{N\in \N_+} \sup_{ \{t_i \} } \|\mathcal O^\varrho_{ \{t_i\}, N}   
(\mathcal H_t  f(x):\,t\in \mathcal I)\|_{L^p(\gamma_\infty)}\lesssim 
\|f\|_{L^p(\gamma_\infty)}.\]
Here and in the rest of the paper
the double $\sup$  denotes the supremum over all increasing
sequences $\left(t_i \right)_0^N$ of points in $\mathcal I$ and all $N \in \N_+$.
The left-hand $L^p$ norm here can be replaced by an $L^{p,\infty}$ (quasi-)norm.

For all $\varrho\ge 1$ and any   increasing sequence  $(t_{i})_{0}^N$ as above,
it is plain that
\begin{equation}\label{eq:vrho_orho}
\mathcal O^\varrho_{ \{t_i\}, N}   (\mathcal H_t  f (x): \,t\in \mathcal I)
 \le    \|\mathcal H_t \,f(x)\|_{v (\varrho), \mathcal I}.
\end{equation}
  On the other hand, oscillation is known to be incomparable with jumps; see \cite{Mirek5}.

By means of \eqref{eq:vrho_orho}, it is straightforward to deduce from items (1), (2), (3) in Subsection \ref{subs:var} 
the first two of   
the following oscillation inequalities for a general Ornstein--Uhlenbeck semigroup:
\begin{enumerate}

 \item[(1'')]
 For $\varrho>2$ and for all  $1<p<\infty$ 
\[\sup_{N\in \N_+} \sup_{ \{t_i \} }
\|\mathcal O^\varrho_{ \{t_i\}, N}   (\mathcal H_t  f(x):\,t\in \R_+)\|_{L^p(\gamma_\infty)}\lesssim \|f\|_{L^p(\gamma_\infty)}.\]
 \item[(2'')]
For $\varrho>2$   also 
\[\sup_{N\in \N_+} \sup_{ \{t_i \} }
\|\mathcal O^\varrho_{ \{t_i\}, N}   (\mathcal H_t  f(x):\,t\in \R_+)\|_{L^{1,\infty}(\gamma_\infty)}\lesssim \|f\|_{L^1(\gamma_\infty)}.\]
\item[(3'')]
For   $\varrho = 2$ and  $1<p<\infty$
\[\sup_{N\in \N_+} \sup_{ \{t_i \} }
\|\mathcal O^2_{ \{t_i\}, N}   (\mathcal H_t  f(x):\,t\in \R_+)\|_{L^p(\gamma_\infty)}\lesssim \|f\|_{L^p(\gamma_\infty)}.\]
 This follows from  \cite[p.2092]{Le Merdy}.  Notice that the notion of oscillation introduced by Le Merdy and Xu is slightly different but equivalent to ours.
\end{enumerate}

\medskip

\begin{remark}
 In the light of a counterexample given in a discrete setting  in \cite[Theorem 1.2]{Mirek5},
  when $1\le p<\infty$ and $1<\varrho\le \varrho'<\infty$, we can expect 
  neither
 \begin{equation*}
\|\,   \|\mathcal H_t \,f(x)\|_{v(\varrho'), \R_+}\|_{L^{p,\infty}}\lesssim_{\varrho,\varrho'} 
\sup_{N\in \N_+} \sup_{ \{t_i \} }
\|  \mathcal O^\varrho_{ \{t_i\}, \infty}   (\mathcal H_t  f:\,t\in\R_+)\|_{L^p(\gamma_\infty)},
\end{equation*}
nor
\begin{equation*}
 J_{\varrho'}^{p,\infty}  (\mathcal H_t  f:\, t\in\R_+)\lesssim_{\varrho,\varrho'}  
 \sup_{N\in \N_+} \sup_{ \{t_i \} }\|  \mathcal O^\varrho_{ \{t_i\}, \infty}   (\mathcal H_t  f:\,t\in\R_+)\|_{L^p(\gamma_\infty)}.
\end{equation*}
It  seems thus that neither the variational nor the jumps can be controlled by the oscillation.
\end{remark}

Regarding oscillation inequalities, 
what is missing to complete the picture is the weak type $(1,1)$
for $\mathcal H_t$ when $\varrho=2$.
The second main result in this paper is exactly
this.      
\begin{theorem}\label{thm:main2}
The following second-order uniform  oscillation inequality holds
for the Ornstein--Uhlenbeck semigroup:
\[\sup_{N\in \N_+} \sup_{ \{t_i \} }
\|\mathcal O^2_{ \{t_i\}, N}   (\mathcal H_t  f(x):\,t\in\R_+)\|_{L^{1,\infty}(\gamma_\infty)}\lesssim \|f\|_{L^1(\gamma_\infty)}.\]
 \end{theorem}
As shown by the authors in \cite{CCS8, CCS9},
some parts of the variation operator 
$ f \mapsto \| \mathcal H_t f\|_{v(\varrho),\Bbb R_+}$
are
of weak type $(1,1)$ not only for $\varrho>2$,  but even  
for $\varrho \ge 1$.  Then  \eqref{eq:vrho_orho} implies that the corresponding parts of 
the oscillation associated with $ \mathcal H_t f$ 
 is of weak type $(1,1)$ as well. Based on these considerations, we shall see that,
 in order to prove Theorem \ref{thm:main2}, it is enough to consider the oscillation of the local part of $\mathcal H_t f$ when $t$ is small.

%%%PRELIMINARIES

\section{Preliminaries} \label{s:prelim}
As in \cite{CCS2}, the symbol
$R(x)$ will denote
the quadratic form 
\begin{equation*}
R(x) ={\frac12 \left\langle Q_\infty^{-1}x ,x  \right\rangle}, \qquad\text{$x\in\R^n$.}
\end{equation*}
We let
\[
|x|_Q = | Q_\infty^{-1/2}x|, 
\]
so  that $R(x) = |x|_Q^2/2$.
Notice that  $|x|_Q$
 is a norm on $\R^n$ and that  $|x|_Q \simeq |x|$.

The invariant measure is given by $d\gamma_\infty(x) = (2\pi)^{-\frac{n}{2}}
(\det \, Q_\infty)^{-\frac{1}{2} }e^{-R(x)} \,dx$.
Further, $D_t$ will be  a one-parameter group  of matrices, defined by
\begin{equation}\label{def:Dtx}
D_t =
 Q_\infty
 e^{-tB^*} Q_\infty^{-1} , \qquad t\in\R.
\end{equation}
We  recall from \cite[Lemma 2.3]{CCS4}
that for $|t|\le 1$ and  $ x \in \R^n$ one has
 \begin{equation}\label{dtx}  |x-D_t \,x|\simeq |t|\,|x|.
       \end{equation}
  Further, 
\begin{equation}\label{3.X}
  |D_{-t}\,x| \lesssim e^{-ct}|x|, \qquad t>0,\;\; x \in \R^n;
\end{equation}
see [9, lemma 2.1].       
 We shall also need 
\cite[formula (4.3)]{CCS2}, that is,
\begin{align}
&\frac{\partial}{\partial s}
R\big( D_s \, x \big)
 =\frac12\,
\big|
Q^{1/2}   Q_\infty^{-1}   D_s\, x\big|^2
\simeq
\big| 
D_s \, x\big|^2,
 \label{vel-4}
\end{align}
for all $x$
in $\R^n$ and $s\in\R$.

Starting from  \eqref{Kolmo},
one immediately obtains
\begin{align*}
\mathcal H_t
f(x)
&=
(\det \, Q_\infty)^{\frac{1}{2} }
(\det \, Q_t)^{-\frac{1}{2} }\\
&\qquad\times \int  
f(u)
\exp\left[
{-\frac12 \langle Q_t^{-1} ( e^{tB}x-u), e^{tB}x-u\rangle}\right] e^{R(u)}
d\gamma_\infty (u),
\end{align*}
and after  some computations also
\begin{align*}
\mathcal H_t
f(x)
&
=
(\det \, Q_\infty)^{\frac{1}{2} } (\det \, Q_t)^{-\frac{1}{2} }  \, e^{R(x)}\\
&\int  
f(u)
\exp\Bigg[
%{-\frac12 \langle Q_t^{-1}  e^{tB}x, e^{tB}x\rangle}
{\frac12 \langle (Q_\infty^{-1}-Q_t^{-1}) (u-D_t\, x), (u-D_t\, x)\rangle}
\Bigg]
d\gamma_\infty (u),
\end{align*}
(see \cite[Section 2]{CCS2}).
Thus,  the Mehler  kernel   $K_t$
may be expressed as 
\begin{align}\label{def:Mehler1}
K_t (x,u)
&=
\Big(
\frac{\det \, Q_\infty}{\det \, Q_t}
\Big)^{{1}/{2} }  \exp
{
\big( R(u)\big)}
\exp\left[
{-\frac12 \langle Q_t^{-1} ( e^{tB}x-u), e^{tB}x-u\rangle}\right],\;
\,
\end{align}
or  as
\begin{align}\label{def:Mehler2}
K_t (x,u)
&=
\Big(
\frac{\det \, Q_\infty}{\det \, Q_t}
\Big)^{{1}/{2} }
\exp
{
\big( R(x)\big)}
\notag\\
&\qquad\qquad\qquad
\times\exp \Big[
{-\frac12 
\left\langle (
Q_t^{-1}-Q_\infty^{-1}) (u-D_t\, x) \,,\, u-D_t\, x\right\rangle}\Big]
\,
\end{align}
 for $x,u\in\R^n$.  From
\eqref{def:Mehler2} it follows  that, for  $0<t\leq 1$ and all $x,u\in  \R^n$,
\begin{equation}\label{litet}
   K_t(x,u)
\,\lesssim \, \frac{ e^{R( x)}}{t^{n/2}} \exp\left[-c\,\frac{|u-D_t\, x |^2}t\right],
\end{equation}
see  \cite[(3.4)]{CCS2}, 
and from \eqref{def:Mehler1} similarly
\begin{equation}\label{litetJ}
  K_t(x,u)
\,\lesssim \, \frac{ e^{R( u)}}{t^{n/2}} \exp\left[-c\,\frac{|e^{tB}x-u |^2}t\right].
\end{equation}

\subsection{Derivative of the Mehler kernel}
We shall also need some information on the $t$ derivative $\dot K_t$  of $K_t$.
For all $(x,u) \in \mathbb R^n\times \mathbb R^n$ and
$t>0$, we proved in \cite[Lemmas 4.2 and 4.3]{CCS5} that
\begin{align}\label{def:Kt_deriv}
\dot K_t(x,u) = K_t(x,u) \, N_t(x,u),
\end{align}
where the function $N_t$ is given by
\begin{align}\label{R}
\notag
N_t (x,u)&=-\frac12\,{\tr
\big(Q_t^{-1} \, e^{tB}\, Q\, e^{tB^*}\big)}
+\frac12\,
\left| Q^{1/2}\, e^{tB^*}\, Q_t^{-1}\,(u-D_t\, x)
\right|^2
\notag\\
&\qquad\qquad\qquad
-
\left\langle Q_\infty \,B^*\, Q_\infty^{-1}\, D_t\, x\,,\,
(Q_t^{-1}-Q_\infty^{-1})\,(u-D_t\, x)\right\rangle,
\end{align}
and therefore satisfies
\begin{align}\label{R1}
|N_t (x,u)|
\lesssim
\frac{1}{t}
+\frac{\left|u-D_t \,x\right|^2}{t^2}
+ |x|\,\frac{|u-D_t\, x|}t
\end{align}
  for $0< t\leq1$.
  Then, by slightly decreasing the positive constant $c$ 
  in the exponential factor $\exp(-c|u-  D_{t}\,x|^2/t)$ of $K_t$ in \eqref{litet}, 
  we can absorb all powers of $|u-D_{t}\,x|^2/t$ appearing in $N_t (x,u)$,
so that
  \begin{equation}\label{dotKeps}
|\dot K_t (x,u)| \lesssim e^{R(x)}\,t^{-n/2}\,\exp\left(-c\,\frac{|u-  D_{t}\,x|^2}t \right)\,
\left(\frac{|x|}{\sqrt t}+\frac1{t}\right), \qquad 0<t\leq 1.
\end{equation}
Analogously,\eqref{litetJ} implies that     
  \begin{equation}\label{dotKeps1}
|\dot K_t (x,u)| \lesssim e^{R(u)}\,t^{-n/2}\,\exp\left(-c\,\frac{|u-  e^{tB}x|^2}t \right)\,
\left(\frac{|x|}{\sqrt t}+\frac1{t}\right), \qquad 0<t\leq 1.
\end{equation}

Moreover, we also have
\begin{align}\label{P1}
|N_t (x,u)|
\lesssim
|D_{-t}\,u- x|\,|D_{-t}\,u|+
e^{-ct}\,
|D_{-t}\,u- x|^2 +e^{-ct}.           
\end{align}
  for $t\ge 1$.

\subsection{Local and global parts of the operators}\label{subs:loc_gl}

  We modify slightly the localization procedure described in \cite[Section 7]{CCS5}, mainly by multiplying the radii of the balls by a large factor $A$.

Let us consider closed balls 
$B_j = B(x_j, A/(1+|x_j|))$, for $j=0,1,\dots,$ 
and some large $A = A(n,Q,B) > 4$ determined later. The interiors of the $B_j$
shall be pairwise disjoint, and the family  $(B_j)_0^\infty$ is maximal with respect to this property.
The scaled balls $3B_j = B(x_j, 3A/(1+|x_j|))$ then cover  $\mathbb R^n$; see \cite[Section 7]{CCS5}. 
The sequence starts with $B_0 = B(0, A)$, which implies that $|x_j|>A$ for $j \ge 1$.

 Then we define  smooth, nonnegative functions $r_j,\;j=0,1,\dots,$ such that 
  $\sum_0^\infty r_j = 1$ in  $\mathbb R^n$, and with  $r_j$ supported in $4B_j$. We further make $r_0 = 1$ in $\frac12B_0$, 
  so that the  $r_j$ with $j>0$ vanish in $\frac12B_0$. 
   The larger  functions 
  $\widetilde r_j,\;j=0,1,\dots,$ will also have values in $[0,1]$, and  $\widetilde r_j$ is supported in  $6B_j$ and equals 1 in $5B_j$.
 For the gradients, one has
 \begin{equation}\label{nabla2}   |\nabla  r_j(x)| ,\;\;
\big|\nabla \widetilde r_j (x)\big|\lesssim 1+|x|.
\end{equation}
  The density of $\gamma_\infty$ is essentially constant in each $6B_j$, since 
  \begin{equation}\label{density} 
  e^{R(x)} \simeq e^{R(x_j)}, \qquad x \in  6B_j,
\end{equation}
uniformly in $j$, as easily verified. The balls  $6B_j$ have bounded overlap, see  \cite[Section 7]{CCS5}. 
 
We then define the localization function as
\begin{align*}
\eta(x,u) = \sum_{j=0}^{\infty} \widetilde r_j (x) \, r_j(u), \qquad x, u \in \R^n.
\end{align*}
This sum is locally finite, and $0 \le \eta(x,u) \le 1$.

The local part of the semigroup 
is now defined for each $t> 0$ as
\begin{align}\label{Hloc} 
\mathcal H_t^{\mathrm{loc}} f(x) 
= \int_{\R^n} K_t(x,u) \, \eta(x,u) \, f(u) \, d\gamma_\infty(u) 
= \sum_{0}^{\infty} \widetilde r_j (x) \, \mathcal H_t (fr_j)(x). 
\end{align} 
The global part is
 $$
 \mathcal H_t^{\mathrm{glob}} = \mathcal H_t - \mathcal H_t^{\mathrm{loc}},
 $$ 
with  kernel $K_t(x,u)(1-\eta(x,u))$.

 \bigskip
                           
 \begin{lemma}\label{lemma3.1}
   If $\eta(x,u) < 1$, then $|x-u| \ge \frac{A}{4(1+|x|)}$.
 \end{lemma}
Thus the support of the kernel of $\mathcal H_t^{\mathrm{glob}} $  is contained in the set
 $$
 G_A := \left\{ (x,u) \in \R^n \times \R^n: |x-u| \ge \frac{A}{4(1+|x|)} \right\}.
  $$

\begin{proof}
Equivalently, we will prove that
\begin{equation}\label{eq:777}
|x-u|<\frac{A}{4(1+|x|)}
\end{equation}
implies $\eta(x,u)=1$. Notice that $\eta(x,u)=1$ follows if $\widetilde{r}_{j}(x)=1$ for any $j$ with $r_{j}(u)>0$.

Starting from \eqref{eq:777}, 
we further assume that $r_{j}(u)>0$ for some $j\in\mathbb{N}$, which implies $u\in4B_{j}$, and verify that $x\in5B_{j}$ so that $\widetilde{r}_{j}(x)=1$. The lemma will then follow.

The triangle inequality, \eqref{eq:777} and the assumption $u\in4B_{j}$ give
\begin{equation} \label{eq:3.9}
|x-x_{j}|\le|x-u|+|u-x_{j}|<\frac{A}{4(1+|x|)}+\frac{4A}{1+|x_{j}|}.
\end{equation}
If $j=0$ so that $x_{j}=0$, \eqref{eq:3.9} implies $|x-x_{j}|<5A=5A/(1+|x_{j}|)$ and we are done.

Assume now that $j>0$; then $|x_{j}|>A$ and $|u|>A/2$,
 the latter since $r_{j}(u)>0$ and $r_{j}$ vanishes in $B(0,A/2)$.
It follows that $|x|\ge|u|-|x-u|>A/2-A/4=A/4,$ and so the right-hand side of \eqref{eq:3.9} is no larger than
\[
\frac{A}{4(1+A/4)}+\frac{4A}{1+A}\le 1+4.
\]
Hence, $1+|x_{j}|\le1+|x|+|x-x_{j}|\le1+|x|+5\le4(1+|x|)$, the last step since $|x|\ge A/4>1$, and thus $(1+|x|)^{-1}\le4(1+|x_{j}|)^{-1}$. From \eqref{eq:3.9} it now follows that
\[
|x-x_{j}|\le\frac{A}{1+|x_{j}|}+\frac{4A}{1+|x_{j}|}=\frac{5A}{1+|x_{j}|};
\]
which means $x\in5B_{j}$. The lemma is proved.
\end{proof}

 For future convenience, we also note that
 for any $\varrho \in [1,\infty)$, any interval $\mathcal I \subset \R_+$ and all $x\in\R^n$
 \begin{align}  \label{important}
 \|\mathcal H_tf(x)\|_{v(\varrho), I}&\le      
     \int_{\mathcal I} \left|\frac{\partial}{\partial t} \int K_t(x,u) f(u)\,d\gamma_\infty (u)  \right|\,dt\notag
 = \int_{\mathcal I} \left| \int \dot K_t(x,u) f(u) \,d\gamma_\infty (u) \right|\,dt     \notag  \\
&\le \int          \int_{\mathcal I} \big|  \dot K_t(x,u)\big| \,dt \, |f(u)|\,  d\gamma_\infty (u).
  \end{align}
The second step here is verified in \cite[formula (3.4)]{CCS8}.

\subsection{Hypercontractivity properties of $(\mathcal H_t)_{t>0}$}  
~
\begin{proposition}\label{p:hyper}
Let $1\le p<q<\infty$.
\begin{itemize}
    \item[$(i)$] If 
    \[
    q -1 \leq (p - 1) \|Q_\infty^{-1/2} e^{tB} Q_\infty^{1/2} \|^{- 2},
    \]
    then $\mathcal H_t$ is a contraction from $L^p(\gamma_\infty)$ into $L^q (\gamma_\infty)$;
    
    \item[$(ii)$] if 
    \[
    q - 1 > (p - 1) \|Q_\infty^{-1/2} e^{tB} Q_\infty^{1/2} \|^{- 2},
    \]
    then $\mathcal H_t$ is not bounded from $L^p (\gamma_\infty)$ into $L^q (\gamma_\infty)$.
\end{itemize}
\end{proposition}
For a proof of this result,  holding also in the infinite-dimensional setting, we refer to  \cite{CM} and  \cite{MF}.

Following the notation in \cite{CM}, we denote by $S_0(t) $ the semigroup
$$
S_0(t)= Q_\infty^{-1/2} e^{tB} Q_\infty^{1/2}, \qquad t>0.
$$
It has been proved in 
 \cite[Propositions 4.2 and  4.1]{MF} that 
\begin{equation}\label{in:MF}
  \|S_0(t) \| <1\qquad \text{ for all $t>0$,}
\end{equation}
and clearly
\begin{equation}\label{in:MF3}
\lim_{t\to +\infty}  \|S_0(t) \| =0.
\end{equation}
In other words, the semigroup $(S_0(t))_{t>0}$ is uniformly stable.
This implies 
that
\[\|S_0(t) \|\lesssim e^{-at}, \quad t>0.\]
We also notice that $\|S_0(t)\|$ is decreasing on $(0,+\infty)$, as a consequence of  the semigroup property and \eqref{in:MF}.
Thus, in particular,
\begin{equation}\label{in:tlargerthan1}
 \|S_0(t)\|\le \|S_0(1)\|<1\qquad\text{for all $t\ge 1$.}
\end{equation}

\section{The strong type $(p,p)$ of jumps for large times}\label{s:strong_type_large_times}
We start by proving that the variation  operator of order $\varrho\ge 1$ associated with  $\left(\mathcal H_t\right)_{t\ge 1}$ 
is bounded on $L^p(\gamma_\infty)$ for all $1<p<\infty$.  Then \eqref{domination} and \eqref{eq:vrho_orho} will imply that 
 the jump  operator   and the oscillation  
of order $\varrho\ge 1$ 
associated with $\left(\mathcal H_t\right)_{t\ge 1}$ have the same boundedness properties.
 \begin{proposition}\label{pr:strongpp_tlarge}
Let  $\varrho \ge 1$.
Then
\begin{equation}\label{(p,p):operator}
 \norm{ \| \mathcal H_t f(x)\|_{v(\varrho), [1,\infty)}}_{L^p(\gamma_\infty)}
 \lesssim \|f\|_{L^p(\gamma_\infty)}
 \end{equation}
for all $f\in L^p(\gamma_\infty)$, $1<p<\infty$.
\end{proposition}
\begin{proof}
Fix $1<p<\infty$.
By \eqref{important} and \eqref{P1} we have
 \begin{align}
& \|\mathcal H_tf(x)\|_{v(\varrho), [1,\infty)}
\le \int         \int_1^{\infty} \big|  \dot K_t(x,u)\big| \,dt \, |f(u)|\,  d\gamma_\infty (u)
\notag\\
&\le \int         \int_1^{\infty}  K_t(x,u) \,|N_t (x,u)|\,dt \, |f(u)|\,  d\gamma_\infty (u)\notag\\
&\le \int         \int_1^{\infty}   K_t(x,u) \Big(|D_{-t}\,u- x|\,|D_{-t}\,u|+
e^{-ct}\,
|D_{-t}\,u- x|^2 +e^{-ct} \Big)\,dt \, |f(u)|\,  d\gamma_\infty (u)\notag.
 \end{align}
 We now write $|D_{-t}\,u - x| \le |D_{-t}\, u\,|+ |x|$ and apply \eqref{3.X} 
several times, 
to conclude that the last expression is dominated by
 \begin{align}
& \int         \int_1^{\infty}   K_t(x,u)\Big(e^{-ct}\,|u|^2+ e^{-ct}\, |x|\,|u|+
e^{-ct}\,
| x|^2 +e^{-ct} \Big)\,dt \, |f(u)|\,  d\gamma_\infty (u)\notag\\
&\lesssim \int         \int_1^{\infty}   K_t(x,u)\, e^{-ct}\,dt\,\Big(|u|^2+
| x|^2+1 \Big)\, |f(u)|\,  d\gamma_\infty (u). 
\label{sum_two_ints}
 \end{align}
We first bound the integral involving the square of $|u|$. By Tonelli's theorem one has

 \begin{align}\label{4.Y}
 \int        \int_1^{\infty}  K_t(x,u) \,e^{-ct}\,dt\, |u|^2 \, |f(u)|\,  d\gamma_\infty (u) &  =     \int_1^{\infty}
\,e^{-ct} \int 
   K_t(x,u) \, |u|^2 \, |f(u)|\,  d\gamma_\infty (u) \,dt    \notag\\
               &    =     \int_1^{\infty}
            e^{-ct} \,\mathcal H_t  \, \big( |\cdot |^2 \, |f(\cdot)|\big)(x)    \,dt.    
              \end{align}

 Since $|u|^2$ is in $L^r(\gamma_\infty)$ for any $r<\infty$, Hölder's inequality implies that $|u|^2 \, |f(u)|$ is in  $L^{p-\varepsilon}(\gamma_\infty)$ for any  $\varepsilon \in (0, p-1]$. If we choose  $\varepsilon$ in this interval and so small that  
 \begin{equation}\label{cond_on_p}
 p-1\le (p-1-\varepsilon)\|S_0(1)\|^{-2},
\end{equation}
then also $p-1\le (p-1-\varepsilon)\,\|S_0(t)\|^{-2}$ for all $t\ge 1$ because of \eqref{in:tlargerthan1}, and  Pro\-position~\ref{p:hyper}\,$(i)$ shows that
 $\mathcal H_t$ is a contraction from $L^{p-\varepsilon}(\gamma_\infty)$ to $L^{p}(\gamma_\infty)$. Thus $\mathcal H_t  \, (|u|^2 \, |f(u)|)(x)$ is in 
 $L^{p}(\gamma_\infty)$, uniformly in   $t\ge 1$.  Minkowski's inequality now shows that the three integrals in \eqref{4.Y} are all in $L^{p}(\gamma_\infty)$.
 
 Next, we consider the integral involving the square of $|x|$ in \eqref{sum_two_ints}.
  Tonelli's theorem now yields
 \begin{align}\label{xx}
 \int         \int_1^{\infty}    K_t(x,u) \,e^{-ct}\,dt\,|x|^2 \, |f(u)|\,  d\gamma_\infty (u)  
   =    |x|^2 \,   \int_1^{\infty}
\,e^{-ct} \mathcal H_t  \, \big(  |f(\cdot)|\big)(x)\,   \,dt.
  \end{align}
  We observe that \eqref{cond_on_p} implies that $p-1+\varepsilon\le (p-1)\|S_0(1)\|^{-2}$, and then  $\mathcal H_t$ is a contraction from $L^{p}(\gamma_\infty)$  to $L^{p+\varepsilon}(\gamma_\infty)$, in view of   Proposition~\ref{p:hyper}\,$(i)$. Thus the right-hand integral in \eqref{xx} is in 
  $L^{p+\varepsilon}(\gamma_\infty)$. Applying  Hölder's inequality almost as before, we conclude that both sides of  \eqref{xx} are in 
  $L^{r}(\gamma_\infty)$ for any  $r<p+\varepsilon$, in particular in  $L^{p}(\gamma_\infty)$.
  
  Since the integral involving the term $+1$ in \eqref{sum_two_ints} is trivial to treat, this ends the proof. 
  \end{proof}

\begin{remark}\label{rem:tlarge}
 It is straightforward to check that 
 Proposition \ref{pr:strongpp_tlarge} holds true even when the time parameter $t$ varies in $[c_0, \infty)$ for some small $0<c_0<1$.
\end{remark}
Inequality \eqref{domination}
and Remark \ref{rem:tlarge}
 immediately imply the following result.

\begin{corollary}\label{c:tlarge}
Let $c_0>0$ be fixed, and let
  $1<p<\infty$.
  Then
\begin{equation*}
 J_2^p  (\mathcal H_t f:\,t\ge c_0)
 \lesssim_{c_0} \|f\|_{L^p(\gamma_\infty)}\,,
 \end{equation*}
and for  $\varrho \ge 1$
 \begin{equation}\label{eq1011}\sup_{N\in \N_+} \sup_{ \{t_i \} }\|
\mathcal O^\varrho_{ \{t_i\}, N}   (\mathcal H_t  f:\,t\ge c_0)\|_{L^{p}(\gamma_\infty)}
\lesssim_{c_0}
    \|f\|_{L^p(\gamma_\infty)}.
\end{equation}
\end{corollary}

\section{The strong type $(p,p)$ of jumps for small $t$  in the global region}
 \label{s:smallt_global}

 In this section, we conclude the proof of Theorem \ref{thm:main1}, proving the following result.

\begin{proposition}\label{propo-glob-tsmall} 
For all $\varrho\ge 1$  and for all $1<p<\infty$,
\begin{equation}\label{ineq:strongtype_j} J_{\varrho}^p \big(\mathcal H^{{\mathrm{glob}}}_t f:\,t\in (0,1]\big)\lesssim \|f\|_{L^p(\gamma_\infty)}.
 \end{equation}
\end{proposition}

Since we know from \cite[Section 9]{CCS9}
that the   jump operator of order $\varrho=2$ associated with the local part of the Ornstein--Uhlenbeck operator for $t\in (0,1]$ is bounded on $L^p(\gamma_\infty)$,
Proposition \ref{pr:strongpp_tlarge}  and Proposition \ref{propo-glob-tsmall}  together yield Theorem  \ref{thm:main1}.

\medskip

The proof of  Proposition \ref{propo-glob-tsmall} occupies the rest of this section.    
Corollary \ref{c:tlarge} implies that we may restrict $t$ to the interval $(0, c_0)$ for some $c_0 \in (0,1)$ that will be determined as we go along.

Combining the estimate
  \eqref{important} with \eqref{domination},  we see 
   that to prove Proposition \ref{propo-glob-tsmall}
it is enough  to verify the $L^p$ boundedness of the operator whose  kernel is\\ $\int_0^{c_0}| \dot K_t(x,u)|\,
   dt \, \chi_{G_A} (x,u) $. 
   
   Further, we need only consider the
case $|x|> 1$. 
In fact, if $|x|\le  1$, then by the global condition $|x-u| \ge A/4 > 1$.
In the light also of \eqref{dtx}, then
 $$
 |u-D_t\,x| \ge |u-x| - |D_t\,x - x|  \ge |u-x| - Ct|x|  \gtrsim 1 
 $$  
provided $c_0$ and thus $t$ are small enough.  
Then \eqref{dotKeps} implies, for all $t < c_0$,
 \[
| \dot K_t(x,u)| \lesssim  e^{R(x)} \, t^{-C}\, \exp(-c/t) \lesssim 1.
 \]
 By means  of \eqref{important}, 
we can then bound  the variation when $|x| \le 1$, 
and therefore also the jumps,                                         
by  $\int |f|\,d\gamma_\infty$ which is bounded by $ \| f \|_{L^p(\gamma_\infty)}$.

 The assumption $|x|> 1$ will be valid in the rest of this section. The globality condition then shows that $|x| |x-u|> 1$, since $A$ is large.

   The kernel $\int_0^{c_0}| \dot K_t(x,u)|\,
   dt \, \chi_{G_A} (x,u) $
 is split  as follows:
   \begin{align*}
   \int_0^{c_0}| \dot K_t(x,u)|\, dt \, \chi_{G_A} (x,u) 
   &= \int_0^{c_0}| \dot K_t(x,u)|\, dt \, \chi_{G_A} (x,u) \,\chi_{R(u)\ge R(x)} (x,u)\\
   &\quad +\int_0^{c_0}| \dot K_t(x,u)|\, dt \, \chi_{G_A} (x,u) \,\chi_{R(u)< R(x)} (x,u),
\end{align*}
                              
  To estimate $\dot K_t(x,u)$ in these two terms, we will apply \eqref{dotKeps} to the first term and \eqref{dotKeps1} to the second.
  In each case, we introduce a partition of $(0,c_0)$ which  depends on the size of the main exponent $|u-D_t\,x|^2/t$ or  $|u-e^{tB}\,x|^2/t$ in 
   \eqref{dotKeps} or \eqref{dotKeps1}, respectively.
   
    The first partition is given by the sets 
\begin{align*}
 I_0(x,u)
=\left\{t \in(0,c_0): \:
   |u-D_t\, x |\leq
  \sqrt t
\,\right\}
\end{align*}
and
\begin{align*}
 I_m(x,u)
=\left\{t \in(0,c_0): \:
 2^{m-1} \sqrt t< |u-D_t\, x |\leq
 2^{m} \sqrt t
\,\right\}, \qquad m = 1,2,\dots.
\end{align*}

The second partition of $(0,c_0)$   consists of  the sets
\begin{align*}
\widetilde I_0(x,u)
=\left\{t \in(0,c_0): \:
   |u-e^{tB} x |\leq
  \sqrt t
\,\right\}
\end{align*}
and
\begin{align*}
\widetilde I_m(x,u)
=\left\{t \in(0,c_0): \:
 2^{m-1} \sqrt t< |u-e^{tB} x |\leq
 2^{m} \sqrt t
\,\right\}, \qquad m = 1,2,\dots.
\end{align*}

\medskip

If  $t \in  I_m(x,u)$, then \eqref{dotKeps} shows that
\begin{align} \label{dotKK}
 | \dot K_t(x,u)|\le  
  e^{R( x)}\
t^{-n/2} \,  \exp\left(-c\,2^{2m}\right) \,
 \left( \frac{|x|}{\sqrt t}+
\frac1{ t}
\right).     
\end{align}
  If instead  $t \in \widetilde I_m(x,u)$, then \eqref{dotKeps1} leads to a similar inequality, where the factor $e^{R( x)}$ is replaced by $e^{R(u)}$.

Set
\begin{equation*}
\mathcal Q_m(x,u) =
  e^{R( x)}
 \int_{I_m(x,u)} t^{-n/2}
 \left(  \frac{|x|}{\sqrt t} +
\frac1{ t}
\right)
dt  \,  \chi_{G_A} (x,u)\,\chi_{R(u)\ge R(x)}(x,u),       \qquad \;
m = 0, 1,\dots,
   \end{equation*}
   and
   \begin{equation*}
\widetilde{\mathcal Q}_m(x,u) =
  e^{R( u)}
 \int_{ \widetilde I_m (x,u)} t^{-n/2}
 \left(  \frac{|x|}{\sqrt t} +
\frac1{ t}
\right)
dt  \,  \chi_{G_A} (x,u)\,\chi_{R(u)< R(x)}(x,u),       \qquad \;
m = 0, 1,\dots.
   \end{equation*}
   
   Observe that the first factor in these two formulas both coincide with $e^{R( x)\wedge R(u)}$.

We define the  kernels
  \begin{equation} \label{suminm}
  \sum_{m=0}^\infty \,  \exp\left(-c\,2^{2m}\right) \,  \mathcal Q_m(x,u)
\end{equation}
and
 \begin{equation} \label{suminmJ}
  \sum_{m=0}^\infty \,  \exp\left(-c\,2^{2m}\right) \,  \widetilde{\mathcal Q}_m(x,u).
\end{equation}

The estimate \eqref{dotKK} implies that  the kernel  $\int_0^{c_0}| \dot K_t(x,u)|\,\, dt\, \chi_{G_A} (x,u)$ can be estimated by the sum of all the kernels in these two sums.
 Proposition \ref{propo-glob-tsmall} will therefore follow from the next proposition, 
 since the factors  $ \exp\left(-c\,2^{2m}\right)$ will                                                                            allow us to sum over $m$ in the space $L^{p}(\gamma_\infty)$

\begin{proposition}\label{Rm} 
Let $m \in \{0, 1,\dots \}$ and $1<p<\infty$. The operators whose kernels are
$\mathcal{Q}_{m} $ and $\widetilde{\mathcal Q}_m$
are both bounded on  $L^p(\gamma_\infty)$,
with a norm bounded by
$C\,2^{Cm}$ for some $C$.
\end{proposition}

Before proving  Proposition \ref{Rm}, we recall some facts from \cite{CCS4}.
From now on, we fix  $m \in  \{0,1,\dots\}$. 
\begin{enumerate}
 \item
If  $ t \in I_m(x,u)$,
then 
\begin{equation} \label{dista}
  |u-x| \le |u - D_t\, x| + |D_t\, x - x| \lesssim  2^m\sqrt t +  t|x|,
\end{equation}
   and also
\begin{align}\label{rtxu}
   |  R(D_t\, x) -R(u)| \lesssim \, (|x|+|u|)\:|D_t\, x - u|_Q 
\lesssim (|x|+|u|)\: 2^m\sqrt t \,.
  \end{align}
  Analogously, if  $ t \in \widetilde I_m(x,u)$,
then 
\begin{equation} \label{distaJ}
  |u-x| \le |u - e^{tB}\, x| + |e^{tB}x- x| \lesssim  2^m\sqrt t +  t|x|,
\end{equation}
since in \eqref{dtx} one can replace $D_t\, x$ by $e^{tB}x$.
Further,
\begin{align}\label{rtxuJ}
   |  R(e^{tB} x) -R(u)| \lesssim \, (|x|+|u|)\:|e^{tB} x - u|_Q 
\lesssim (|x|+|u|)\: 2^m\sqrt t \,.
  \end{align}

 \item Let $(x,u) \in G_A$.
  If $A$ is chosen large enough, depending only on  $n$,   $Q$ and  $B$,
then
\begin{equation}\label{im}
  I_m(x,u) \subset (2^{-2m}/|x|^2,\: c_0), \qquad m = 0,1, \dots ,
\end{equation}
(see \cite[Lemma 9.3]{CCS4}).
Similarly, one has
\begin{equation}\label{imJ}
 \widetilde  I_m(x,u) \subset (2^{-2m}/|x|^2,\: c_0), \qquad m = 0,1, \dots,
\end{equation}
because \eqref{distaJ} allows us to replace $D_t\,x$ by $e^{tB}\,x$ in \cite[Lemma 9.3]{CCS4} and its proof.

                \item
Let  $ t \in I_m(x,u)$.
If the constant  $C_0 > 4$  is chosen large enough,
depending only on $n$,   $Q$ and  $B$,
then  $t > C_0\,2^{2m}/|x|^2$
 implies
\begin{align}
  |u|  &\simeq |x|,\label{u,x}
\\  
   R(u) -  R(x) &\simeq t|x|^2  \simeq |u-x|\,|x|
    \label{Ru-Rx} \\
\intertext{and} 
 t &\simeq  {|u-x|}/{|x|}.  \label{u-x}
\end{align}
We refer to \cite[Lemma 9.4]{CCS4} for a proof.
\end{enumerate}

As we shall now see,  the properties in (3) hold also when $t\in \widetilde I_m(x,u)$,  with one modification.

\begin{lemma}  \label{size}
Let  $t\in \widetilde I_m(x,u)$.
If the constant  $C_0 > 4$  is chosen large enough,
depending only on $n$,   $Q$ and  $B$,
then  $t > C_0\,2^{2m}/|x|^2$
 implies \eqref{u,x}, \eqref{u-x}, and 
\begin{align} \label{Rx-Ru}
   R(x) -  R(u) &\simeq t|x|^2  \simeq |u-x|\,|x|. 
    \end{align}
\end{lemma}

\begin{proof}
We essentially follow the arguments used for \cite[Lemma 9.4]{CCS4}.
Because of our assumptions on $t$,  \eqref{distaJ}  leads to    
\begin{equation}\label{diff}
  |u-x| \lesssim t|x|,
\end{equation} 
and choosing $c_0$ and thus also $t$ small we conclude that
$|u-x| < |x|/2 $. This implies \eqref{u,x}.

In contrast with \eqref{vel-4}, the derivative $\partial_s R(e^{sB}\,x)$ is negative. To estimate it,
 we recall  the identity
\begin{equation*}
\langle  B^* Q_\infty^{-1} x, 
x\rangle =
-
\frac12\, |Q^{1/2}\, Q_\infty^{-1} x|^2
              \end{equation*}
from \cite[Lemma 4.1]{CCS2}.
By replacing here  $x$ by $e^{sB}x$ after using the symmetry of $Q_\infty^{-1}$, we obtain for $0<s<1$
\begin{align}
-\frac{\partial}{\partial s}
R\big( e^{sB}x \big)
&=
-\frac12
\frac{\partial}{\partial s}
\langle Q_\infty^{-1} e^{sB}  x, e^{sB}  x
\rangle
=
-\langle Q_\infty^{-1} e^{sB} x,  B e^{sB} x
\rangle
\notag\\
&=-\langle B^* Q_\infty^{-1} e^{sB} x,  e^{sB} x
\rangle
=
\frac12
\big|
 Q^{1/2} 
 Q_\infty^{-1} e^{sB}  x\big|^2 
 \simeq | x|^2.    \label{der:neg}
 \end{align}
Thus
\begin{equation} \label{diff_R_t}
  R(x) - R(e^{tB}x)  = -\int_0^t \frac{\partial}{\partial s} R(e^{sB}x) \,ds \simeq t|x|^2,
\end{equation}
the last step since $t\in(0,c_0)$.

On the other hand, from \eqref{rtxuJ},                  
\eqref{u,x} and the assumption $t > C_0\,2^{2m}/|x|^2$, we have
\begin{equation*}
\big| R(e^{tB}\, x) -R(u)\big| \lesssim |x|\,2^m\sqrt t \le \frac{1}{\sqrt{C_0}} \,t|x|^2.
\end{equation*}
Writing $R(x) - R(u) = \big(R(x) - R(e^{tB}x)\big) + \big(R(e^{tB}x) - R(u)\big)$, we
can choose  $C_0$ so large that $ \big|R(e^{tB}x) - R(u)\big|$ is smaller than $\big(R(x) - R(e^{tB}x)\big)/2$. Hence, 
$R(x) - R(u) \simeq t|x|^2$.

To conclude we write, using \eqref{diff} for the last step,
               $$
t|x|^2  \simeq R(x) - R(u)  = \frac12 (|x|_Q + |u|_Q) (|x|_Q - |u|_Q) \lesssim |x| |x-u| \lesssim  t|x|^2,
$$ 
 which implies $t \simeq |u-x|/|x|$ and \eqref{Rx-Ru}.
    \end{proof}

Motivated by   items (2) and (3),
we  write  $I_m(x,u)=  I_m^-(x,u) \cup   I_m^+(x,u) $, with 
\begin{equation*}
   I_m^-(x,u) =  I_m(x,u) \cap (2^{-2m}/|x|^2,\:c_0\wedge  C_0\,2^{2m}/|x|^2 ]
\end{equation*}
and
\begin{equation*}       I_m^+(x,u) =  I_m(x,u) \cap (c_0\wedge  C_0\,2^{2m}/|x|^2,\: c_0 ).
\end{equation*}
Similarly, we shall write  $ \widetilde I_m(x,u)=  \widetilde I_m^-(x,u) \cup   \widetilde I_m^+(x,u) $, with analogous definitions of $ \widetilde I_m^\pm (x,u)$.

We then split the kernel $\mathcal Q_m$ as 
 $\mathcal Q_m = \mathcal Q_m^- + \mathcal Q_m^+$   for $m = 0,1,\dots$, where
\begin{align*}
\mathcal Q_m^-(x,u) 
&= e^{R( x)}\int_{I_m^-(x,u)} t^{-n/2} \left( \frac{|x|}{\sqrt t} + \frac1{ t} \right) dt \, \chi_{G_A} (x,u)\,\chi_{R(u)\ge R(x)}(x,u), \\
\intertext{and}
\mathcal Q_m^+(x,u) &=e^{R( x)}\int_{I_m^+(x,u)} t^{-n/2}
 \left( \frac{|x|}{\sqrt t} + \frac1{ t} \right) dt \, \chi_{G_A} (x, u)\, \chi_{R(u)\ge R(x)}(x,u).
\end{align*}
Analogously, we split $\widetilde{\mathcal Q}$ as $\widetilde{\mathcal Q} = \widetilde{\mathcal Q}_m^- + \widetilde{\mathcal Q}_m^+$ by restricting the  integration in $t$ to $ \widetilde I_m^-(x,u)$ and $ \widetilde I_m^+(x,u)$, respectively.

The following simple estimate will be useful when we estimate $\mathcal Q_m^-(x,u)$ and $\widetilde{\mathcal Q}_m^-(x,u)$: 
\begin{equation}\label{simpleint}
   \int_{2^{-2m}/|x|^2}^{C_0\,2^{2m}/|x|^2} t^{-n/2} \,
 \left( \frac{|x|}{\sqrt t}  +
\frac1{ t}
\right)\, dt  \lesssim 2^{mn} \,|x|^n.
\end{equation}
Notice that the upper bound of integration matters only for $n=1$.

\begin{lemma}\label{tipoforte}
Let $1\le p\le \infty$.
The operator with kernel
$\mathcal Q_m^-(x,u) $
 is of strong type $(p,p)$ with respect to $d\gamma_\infty$, with a norm bounded by
$C\,2^{Cm}$.
\end{lemma}

\begin{proof}
In \cite[Lemma 9.5]{CCS4} we showed that this operator  is of strong type $(1,1)$ with respect to $d\gamma_\infty$, with a norm bounded by
$C\,2^{Cm}$. 

We will  prove that it is similarly bounded on $L^\infty (\gamma_\infty)$.

 For $ t < C_0\,2^{2m}/|x|^2$, the estimate \eqref{dista} implies
  \begin{equation} \label{qm-}
  |u-x| \le 2\,C_0\, 2^{2m}/|x|.
  \end{equation}
  
  Together with \eqref{simpleint} and the definition of $\mathcal Q_m^- $, this leads to
\begin{align*}
\mathcal Q_m^-(x,u)
\lesssim  \ e^{R( x)\wedge R(u)}\,2^{Cm}\,|x|^{n}\,
\chi_{\left\{|u-x|\le  2\,C_0\, 2^{2m}/|x| \right\} }.
  \end{align*}

Thus
\begin{align*}
  \int \mathcal Q_m^-(x,u)\, \,  d\gamma_\infty (u) &\lesssim
     2^{Cm}\,|x|^{n}\,
\int e^{R( x)\wedge R(u)}\,\chi_{\left\{|u-x|\le  2\,C_0\, 2^{2m}/|x| \right\}}  \,d\gamma_\infty (u)\\
& \lesssim  \ 2^{Cm}\,|x|^{n}\,
\int_{\left\{|u-x|\le  2\,C_0\, 2^{2m}/|x| \right\}}\, du
 \lesssim  \ 2^{Cm},
 \end{align*}
which proves the desired boundedness on $L^\infty (\gamma_\infty)$.

A standard interpolation argument now completes the proof of the lemma.
\end{proof}

\medskip

Analogous $L^p$ mapping properties hold for  $\widetilde{\mathcal{Q}}^{-}_{m}$, as we shall now see.

\begin{lemma}\label{tipofortetilde}
Let $1\le p\le \infty$.
The operator with kernel
$\widetilde{\mathcal{Q}}^{-}_{m}(x,u) $
 is of strong type $(p,p)$ with respect to $d\gamma_\infty$, with a norm bounded by
$C\,2^{Cm}$.
\end{lemma}

\begin{proof}

We first verify the case $p=1$.
This can be done  as in  \cite[Lemma 9.5]{CCS4}, by  replacing  $I_m^-$ by $ \widetilde I_m^-$.
Here we give a simplified version of 
the proof.

 If $t  \in \widetilde{\mathcal{I}}^{-}_{m}(x,u)$ so that $ t < C_0\,2^{2m}/|x|^2$, the estimate \eqref{distaJ} implies that \eqref{qm-} holds again,
  and also, quite trivially, that
   \begin{equation} \label{trivest}
  |u-x| \lesssim  2^m\sqrt t +  t|x| \lesssim 2^m + \sqrt t\, |x|  \lesssim 2^m.
  \end{equation}
  Hence,
   \begin{equation} \label{trivx}
 |x| \le |u| +  |u-x| \lesssim |u | +  2^m. 
  \end{equation}
 
                         The expression for   $\widetilde{\mathcal{Q}}^{-}_{m}(x,u) $   and \eqref{simpleint} together  yield                  
 \begin{align*}
  \int \widetilde{\mathcal Q}_m^-(x,u)\, &\,  d\gamma_\infty (x) \lesssim
  2^{Cm}\,
\int_{R(u) < R(x)}  e^{R( x)\wedge R(u)}\,|x|^{n}\,\chi_{\left\{|u-x|\le  2\,C_0\, 2^{2m}/|x| \right\}}  \,d\gamma_\infty (x). 
 \end{align*}
 Here we estimate the factor $|x|^{n}$ by means of  \eqref{trivx}. Further,
  $R(u) < R(x)$ implies $|u|\simeq |u|_Q < |x|_Q \simeq |x|$, so that $|x|^{-1} \lesssim |u|^{-1}$. If we combine this with \eqref{trivest}, we obtain
   \begin{align*}
  \int \widetilde{\mathcal Q}_m^-(x,u)\, \,  d\gamma_\infty (x) & \lesssim
   2^{Cm}\,\left(|u|^{n}+ 2^{nm}\right)
   \int_{|x-u|\le C2^{m}\wedge C 2^{2m}/|u|}   \,dx  \\
   & \lesssim
   2^{Cm}\,\left(|u|^{n}+ 2^{nm}\right) 
   \left( 2^{mn}\wedge C  2^{2mn}/|u|^n    \right)
\\
  & \lesssim 2^{Cm}\,  \left(\frac{2^{2nm}|u|^{n}}{|u|^{n}} + 2^{2nm}\right) 
  \lesssim  2^{Cm}.
 \end{align*}
  
  The desired bound for $p=1$ follows.

For the case $p= \infty$, we apply again \eqref{simpleint} and \eqref{qm-} to get
\begin{align*}
  \int  \widetilde{\mathcal{Q}}^{-}_{m}(x,u)\,   d\gamma_\infty (u) \lesssim
2^{Cm}\, |x|^{n}\,
\int_{\left\{|u-x|\lesssim  2\,C_0\, 2^{2m}/|x| \right\}} \, du \lesssim  \ 2^{Cm}.
 \end{align*}
An interpolation now ends the proof of Lemma~\ref{tipofortetilde}.
\end{proof}

\medskip
Next, we deal with the  kernel
$\mathcal{Q}^{+}_{m} $.

\begin{proposition}\label{tipofortepp}
The operator with kernel
$\mathcal{Q}^{+}_{m}(x,u) $
 is bounded on  $L^p(\gamma_\infty)$, for all $1< p\le \infty$,  with a norm bounded by
$C\,2^{Cm}$.
\end{proposition}
\begin{proof}
In  \cite[Lemma 9.6]{CCS4} it is proved that the kernel
\begin{align}
\mathcal D_m^+(x,u) &=e^{R( x)}\int_{I_m^+(x,u)} t^{-n/2} \left(|x|^2+  \frac1{ t} \right) dt \, \chi_{G_A} (x, u)
\end{align}
gives an operator of weak type (1,1) with respect to $d\gamma_\infty$, with quasinorm at most $C2^{Cm}$.

By means of the inequality between geometric and arithmetic means, we see that 
$ \mathcal Q_m^+\lesssim  \mathcal D_m^+$, 
so that
$\mathcal Q_m^+$ defines an operator with the same weak type property.

We will soon verify that the operator with kernel $\mathcal Q_m^+ (x,u)$ is also bounded on $L^\infty (\gamma_\infty)$, with  norm at most 
$C\,2^{Cm}$.
By interpolation, the proposition then follows.

In the light of \eqref{rtxu}, \eqref{u,x} and \eqref{u-x},
 any $t\in I_m^+(x,u)$ satisfies
\begin{align} \label{restr-t}
\big|R(D_t \,x) -R(u)\big|
\lesssim |x|\,2^m\,\sqrt t
\simeq  2^m\, \sqrt {{|x-u|}{|x|}}.
\end{align}
Since \eqref{vel-4} implies  that
\begin{equation*}
  \partial_t R(D_t\, x)
\simeq |x|^2, \quad 0<t<1,
\end{equation*}

we deduce that \eqref{restr-t} can  hold 
for $t\in I_m^+(x,u)$  only when $t$ is       in
a subinterval of length at most $ C\,  2^m\, \sqrt{|x-u|}/ |x|^{3/2}$,
call it $I$.

 Starting with the definition of $\mathcal Q_m^+(x,u)$, we apply first \eqref{u-x} to replace $t$ by ${|x-u|}/{|x|}$, 
 then  the bound on the length of $I$, and finally the gobality condition ${|x-u|}{|x|}>1$:
   \begin{align} \label{est_tildeQ} &
{\mathcal{Q}}^{+}_{m}(x,u) 
\lesssim   e^{R( x)}\,  2^m\, \frac{ \sqrt{|x-u|}} {|x|^{3/2}}\  \,
\left(  \frac{|x-u|}{ |x|} \right)^{-n/2} \, \left( \frac{|x|^{3/2}}{ \sqrt{|x-u|}} +  \frac{|x|}{ |x-u|} \right)\notag \\ 
&\qquad = e^{R( x)}\,  2^m\, \left(  |x|^{n/2}   |x-u|^{-n/2}    +    |x|^{(n-1)/2}    |x-u|^{-(n+1)/2}    \right) 
\lesssim   e^{R( x)}\, 2^m\, |x|^{n}.\qquad
\end{align}

  Observe  that if $\mathcal Q_m^+(x,u) \ne 0$, then there exists a $t \in  I_m^+(x,u)$, and   \eqref{Ru-Rx} says that $R(u)-R(x) \simeq |x||x-u|$.
    We can now integrate $\mathcal Q_m^+(x,u)$, to get 
    \begin{align*}
 \int \mathcal Q^+_m(x,u) \,d\gamma_\infty(u)&
 \lesssim 2^m\, |x|^{n}\,\int_{R(u) - R(x) >c|x||x-u|} e^{R(x)-R(u)}\,du \\
& \lesssim 2^m\, |x|^{n}\,\int e^{-c|x|\,|x-u|}\,du \simeq 2^m.
\end{align*}

Thus the operator with kernel
$ \mathcal{Q}_{m}^+ (x,u)$
 is bounded on $L^\infty (\gamma_\infty)$, with a  norm  at most
$C\,2^{m}$. 
\end{proof}

It remains to treat $\widetilde{\mathcal{Q}}^{+}_{m}$.

\begin{proposition}\label{tipofortepp_tilde}
The operator with kernel
$\widetilde{\mathcal{Q}}^{+}_{m}(x,u)$
 is bounded on  $L^p(\gamma_\infty)$ for all $1 < p < \infty$, with a norm bounded by
$C\,2^{Cm}$.
\end{proposition}

\begin{proof}
We start by copying the arguments that led to \eqref{est_tildeQ} in the proof of Proposition \ref{tipofortepp}, with only small modifications.

Using \eqref{rtxuJ} and then \eqref{u,x}  and \eqref{u-x}  (these two are still valid here), we see that
 any $t\in \widetilde I_m^+(x,u)$ satisfies
\begin{align} \label{restr-tJ}
\big|R(e^{tB} x) -R(u)\big|
\lesssim |x|\,2^m\,\sqrt t
\simeq  2^m\, \sqrt {{|x-u|}{|x|}}.
\end{align}
Since \eqref{der:neg} says that $-\partial_t R(e^{tB} x) \simeq |x|^2$, it follows that
 \eqref{restr-tJ} can  hold 
for $t\in  \widetilde I_m^+(x,u)$  only when $t$ is       in
a subinterval of length at most $ C\,  2^m\, \sqrt{|x-u|}/ |x|^{3/2}$.
    
 We apply this to the integral in the expression for $\widetilde{\mathcal{Q}}^{+}_{m}(x,u) $, after  replacing $t$ by ${|x-u|}/{|x|}$ by means of \eqref{u-x}.
 The result will be like \eqref{est_tildeQ}, except that the first factor is now $e^{R( u)}$; thus we have
   \begin{align} \label{est_tildeQJ}
\widetilde{\mathcal{Q}}^{+}_{m}(x,u) \lesssim   e^{R( u)}\, 2^m\, |x|^{n}.
\end{align}

  In order to conclude, we apply a weighted version of  Schur's test on $L^p(\gamma_\infty)$;  see \cite[Theorem 3.2.2]{Zhu}.
It is enough to prove the following two estimates:
\begin{align}
\int \widetilde{\mathcal{Q}}_m^+(x,u)\, w(u)^{p'}\, d\gamma_\infty(u) &\le C\, 2^{Cm} w(x)^{p'} \qquad \text{for a.e. } x, \label{schur1}\\
\int \widetilde{\mathcal{Q}}_m^+(x,u) \, w(x)^{p} \,d\gamma_\infty(x) &\le C\, 2^{Cm} w(u)^{p} \qquad \text{for a.e. } u, \label{schur2}
\end{align}
where $w(x)=\exp(\alpha R(x))$, with  $0 < \alpha < 1/p$.   

  We start with \eqref{schur1} and use \eqref{Rx-Ru} and \eqref{est_tildeQJ} to get
  \begin{align} \label{schur11} 
\int \widetilde{\mathcal{Q}}_m^+(x,u)\, w(u)^{p'} d\gamma_\infty(u) 
&\lesssim   2^m\, |x|^{n}\, \int_{R(x) - R(u) >c|x||x-u|}  e^{R( u)}\, e^{p'\alpha R( u)}\, d\gamma_\infty(u)\\
 & \lesssim  2^m\,e^{p'\alpha R( x)}\,|x|^{n}\, \int_{R(x) - R(u) >c|x||x-u|} e^{-p'\alpha (R( x)-R( u))}\,du \\
 &\lesssim 2^m\,e^{p'\alpha R( x)}\,|x|^{n}\, \int  e^{-cp'|x||x-u|}\,du \lesssim  2^m\,w(x)^{p'}.
\end{align}
  
  To deal with  \eqref{schur2}, we first observe that if $t \in \widetilde{\mathcal{I}}_m^+(x,u) \ne \emptyset$, then \eqref{u,x} and \eqref{Rx-Ru} hold, 
  i.e.,   $|x| \simeq |u|$ and  $R(x) - R(u) \simeq |u||x-u|$. Instead of \eqref{schur11}, we now get 
   \begin{align}
\int \widetilde{\mathcal{Q}}_m^+(x,u)\, w(x)^{p} d\gamma_\infty(x) 
&\lesssim   2^m\, |u|^{n}\, \int_{R(x) - R(u) >c|u||x-u|}  e^{R( u)}\, e^{p\alpha R( x)}\, e^{-R( x)}\,dx\\
 & \lesssim  2^m\,e^{p\alpha R( u)}\,|u|^{n}\, \int_{R(x) - R(u) >c|u||x-u|} e^{-(1-p\alpha)(R( x)-R( u))}\,dx \\
 &\lesssim 2^m\,e^{p\alpha R( u)}\,|u|^{n}\, \int  e^{-c(1-p\alpha)|u||x-u|}\,dx \lesssim  2^m\,w(u)^{p}.
\end{align} 
The proposition is proved. 
\end{proof}

\medskip

Finally,  
Lemma \ref{tipoforte}, Lemma \ref{tipofortetilde}, Proposition \ref{tipofortepp}
and Proposition \ref{tipofortepp_tilde} together
yield the proof of Proposition \ref{Rm},  and Proposition \ref{propo-glob-tsmall} follows.

\section{The oscillation bounds}\label{s:oll}
In this section we prove Theorem \ref{thm:main2}.

\medskip

\noindent{\em{Proof of Theorem \ref{thm:main2}.}}
Theorems 4.3 and 6.1 from \cite{CCS8} state that for $1\le \varrho <\infty$, the operators mapping $f \in L^1(\gamma_\infty)$ to the functions
\begin{equation}\label{op:oper1}
 \| \mathcal H_t f(x)\|_{v(\varrho), [1,+\infty)}, \quad x \in \mathbb R^n,
\end{equation}
and
\begin{equation}\label{op:oper2}
 \|{ \mathcal H}_t^{\mathrm{glob}} f(x)\|_{v(\varrho), (0,1]}, \quad x \in \mathbb R^n,
\end{equation}
 are both of weak type $(1,1)$ with respect to the measure $d\gamma_\infty$. 

Thus, \eqref{eq:vrho_orho} implies that
\begin{equation}\label{eq101}
\sup_{N\in \N_+} \sup_{ \{t_i \} }\| \mathcal O^\varrho_{ \{t_i\}, N} (\mathcal H_t f:\,t\ge 1)\|_{L^{1,\infty}(\gamma_\infty)} \le \|f\|_{L^1(\gamma_\infty)}
\end{equation}
and
\begin{equation}\label{eq102}
\sup_{N\in \N_+} \sup_{ \{t_i \} }\| \mathcal O^\varrho_{ \{t_i\}, N} (\mathcal H_t^{\mathrm{glob}} f:\,t\in (0,1])\|_{L^{1,\infty}(\gamma_\infty)} \le \|f\|_{L^1(\gamma_\infty)}.
\end{equation}

To conclude the proof, it suffices to prove a uniform  oscillation inequality for $\mathcal H_t^{\mathrm{loc}} f$ with $t\in (0,1]$.
In the sum in \eqref{Hloc} (which is locally finite, since the $\widetilde r_j $ have bounded overlap), we may deal with each term $\widetilde r_j (x) \, \mathcal H_t (f\,r_j)(x)$  separately. Thus we fix  $j \in \N$, and observe that it does not matter whether we use $d\gamma_\infty$ or Lebesgue measure as basic measure,  since they are equivalent in the support of $\widetilde r_j $.
 
We know from \cite[Proposition 6.1]{CCS9} that  the variation operator of $\widetilde r_j  \, \mathcal H_t (f\,r_j)$   in the interval 
 $\min(1,|x_j|^{-2}) < t \le 1$ is of strong type $(p,p)$ for all $p \ge 1$. As a consequence,  we immediately obtain a strong type $(1,1)$ oscillation inequality in the same interval,
and this is uniform in $j$.    

What remains to prove is thus
 a weak type $(1,1)$ inequality for the oscillation 
$ \mathcal{O}^2_{ \{t_i\}, N} (\widetilde{r}_j \mathcal H_t(f r_j))$ in the interval $ 0 < t \le \min(1, |x_j|^{-2}) $.
To do so, we follow the approach introduced in \cite[Section 7]{CCS9},  decomposing the kernels as there.

We define the  difference operators $\Delta_t^{(\kappa)}$ and the main operator $M_t$ as
\begin{align}\label{def_Deltat}
\Delta_t^{(\kappa)} g(x) &= \int_{\mathbb{R}^n} \left(\widetilde{K}_t^{(\kappa-1)}(x,u) - \widetilde{K}_t^{(\kappa)}(x,u)\right)\,\eta(x,u)\, g(u) \, d\gamma_\infty(u), \quad \kappa=1,2,3,\qquad\\
M_t g(x)& = \int_{\mathbb{R}^n} \widetilde{K}_t^{(3)}(x,u) \,\eta(x,u)\, g(u) \, d\gamma_\infty(u),\label{def_Mt}
\end{align}
where the kernels $\widetilde K_t^{(\kappa)}(x,u)$, for $\kappa=0,1,2,3$, are given, respectively, by
 \begin{align}
 \widetilde K_t^{(0)}(x,u)&= K_t(x,u),\notag\\
 \widetilde K_t^{(1)}(x,u)&= \Big(\frac{\det Q_\infty}{\det Q_t}\Big)^{1/2}\, e^{R(x)}\, \exp\Big( -\frac12\, \langle (Q_t^{-1} -Q_\infty^{-1})(u- x), \,u- x \rangle \Big) ,\label{def:tildeKt} \qquad\\
 \widetilde K_t^{(2)}(x,u)&= \Big(\frac{\det Q_\infty}{\det Q_t}\Big)^{1/2}\, e^{R(x)}\, \exp\left( -\frac1{2t}\, \big|Q^{-1/2}(u- x)\big|^2 \right), \label{def:tildeKt_two} \\
 \widetilde K_t^{(3)}(x,u)&= \frac{(\det Q_\infty)^{1/2}}{ (\det Q)^{1/2}\, t^{n/2}}\, e^{R(x)}\, \exp\left( -\frac1{2t}\, \big|Q^{-1/2}(u- x)\big|^2 \right) . \label{def:tildeKt_three} 
 \end{align}
 Then we have
\begin{align}\label{eq:sum_Ht}
\widetilde r_j (x) \, \mathcal H_t (f\,r_j)(x) =
  \widetilde r_j(x) \left( \sum_{\kappa=1}^3 \Delta_t^{(\kappa)} (f r_j)(x) + M_t (f r_j)(x) \right).
\end{align}

Propositions 6.1 and 8.4 in \cite{CCS9} together 
establish that for any $1\le \varrho <\infty$, 
the variation operator associated with $\widetilde r_j(x) \Delta_t^{(\kappa)} (f r_j)(x)$ for $t\in (0,1]$ 
is of strong type $(p,p)$ for all $p \ge 1$ 
with respect to the invariant (and Lebesgue) measure, uniformly in $j$, for  $\kappa=1,2,3$. 
Once again, in the light of \eqref{eq:vrho_orho}, the same then holds for the corresponding oscillation.

We are therefore left with the oscillation estimate for 
 $\widetilde{r}_j M_t(f r_j)$
 in $0 < t \le \min(1, |x_j|^{-2})$.

\begin{proposition}\label{prop:Mt}
Let $j \in \mathbb{N}$ and $1<p<\infty$. Then the following uniform oscillation inequalities hold:
\begin{equation}\label{stima_pp_Mt}
\sup_{N\in \N_+} \sup_{ \{t_i \} }  \|\mathcal O^2_{ \{t_i\}, N} (\widetilde r_j M_t(f r_j): 0 < t \le \min(1, |x_j|^{-2}))\|_{L^p(\gamma_\infty)} \lesssim
 \|f r_j\|_{L^p(\gamma_\infty)}, 
\end{equation}
 and
 \begin{equation}\label{stima_1inf_Mt}
\sup_{N\in \N_+} \sup_{ \{t_i \} } \|\mathcal{O}^2_{ \{t_i\}, N} \big(\widetilde{r}_j M_t(f r_j): 0 < t \le \min(1, |x_j|^{-2}))\|_{L^{1,\infty}(\gamma_\infty\big)}
 \lesssim \|f r_j\|_{L^1(\gamma_\infty)}.
\end{equation}
 The implicit constants are uniform in $j$.

\end{proposition}

\begin{proof}
We fix $j \in \N$ and let $(t_i)_{i=0}^N$ be an arbitrary finite increasing sequence in $(0, \min(1, |x_j|^{-2})]$. 
In this proof, 
we will omit the restriction $0 < t \le \min(1, |x_j|^{-2})$
 inside the oscillation  and simply write $\mathcal{O}^2_{\{t_i\}, N}(\widetilde{r}_j M_t(f r_j))$.

Exactly as in \cite[Proposition 9.1]{CCS9}, for $t>0$ one defines the kernel
\[
\psi_t(y)=\Big(\frac{
 \det Q_\infty}{\det Q}\Big)^{1/2}\,  
{t^{-n/2}}\,\exp\big(-\frac1{2t}\,|Q^{-1/2}\,y|^2\big), \qquad y\in\R^n.
\]
             Then
\begin{align*} 
 M_{t}(f\,r_j)(x)&=   e^{R(x)}
\Big( \psi_t* \big(fr_je^{-R(\cdot)}\big)\Big)(x).
\end{align*}
 for $f \in L^1(\gamma_\infty)$, and  
\begin{align}
\mathcal O^2_{ \{t_i\}, N} (\widetilde{r}_j M_t(f r_j))(x) 
&\lesssim \mathbbm{1}_{6B_j}(x) \,e^{R(x)}\, \mathcal O^2_{ \{t_i\}, N} (\psi_t * (f r_j e^{-R(\cdot)}))(x).\label{ineq:Liu_Osc}
\end{align}
We recall that  $e^{R(x)}\simeq e^{R(x_j)}$ for  $x \in 6 B_j \supset \mathrm{supp} \, \widetilde r_j $.

 Aiming at \eqref{stima_1inf_Mt}, we see that  for any $\alpha>0$,
\begin{align*}
& \gamma_\infty\left\{ x :\,
\mathcal O^2_{ \{t_i\}, N} (\widetilde{r}_j M_t(f r_j))(x)>\alpha\right\}\\
 & \le
\gamma_\infty\left\{ x\in 6 B_j:\,   
e^{R(x_j)}\, \mathcal O^2_{ \{t_i\}, N} (\psi_t * (f r_j e^{-R(\cdot)}))(x)>c\alpha \right\}    
 \\&\simeq
 e^{-R(x_j)}
\left|\left\{ x\in 6 B_j:\,  
 \,\mathcal O^2_{ \{t_i\}, N} (\psi_t * (f r_j e^{-R(\cdot)}))(x)> \, c\, e^{-R(x_j)}\, \alpha \right\}\right|.
\end{align*}
            
Now we  apply  Theorem 2.7 in \cite{Liu} to the last expression, getting 
\begin{align*}
 \gamma_\infty&\left\{ x :\,
\mathcal O^2_{ \{t_i\}, N} (\widetilde{r}_j M_t(f r_j))(x)>\alpha\right\}\\
&\qquad\qquad \lesssim  \,
\frac{ e^{-R(x_j)}}{ e^{-R(x_j)}\,\alpha} \,
\left\|\psi_t *\big( fr_je^{-R(\cdot)}\big)\right\|_{L^1(du)}
 \lesssim \frac{1}{\alpha}\, \|f\,r_j\|_{L^1(\gamma_\infty)},
 \end{align*}
where, exactly as in \cite[Proposition 9.1]{CCS9}, we also exploited 
the fact that convolution with $\psi_t$ defines a bounded operator on $L^1(du)$. 
 Thus
  \[\|\mathcal{O}^2_{ \{t_i\}, N} (\widetilde{r}_j M_t(f r_j))\|_{L^{1,\infty}(\gamma_\infty)} \lesssim \|f r_j\|_{L^1(\gamma_\infty)}.\]
We can now take the supremum over the sequences  $(t_i)_{i=0}^N$, 
 and \eqref{stima_1inf_Mt} is proved.

Next, we prove \eqref{stima_pp_Mt} for $1<p<\infty$. We start from \eqref{ineq:Liu_Osc} and move to Lebesgue measure, thus obtaining
\begin{align*} 
\norm{\, \mathcal O^2_{ \{t_i\}, N} (\widetilde{r}_j M_t(f r_j)) \,}^p_{L^p(\gamma_\infty)}
&\lesssim \int_{6B_j}\Big|
{ \,e^{R(x)}\, \mathcal O^2_{ \{t_i\}, N} (\psi_t * (f r_j e^{-R(\cdot)}))(x)
}\Big|^p \,{ d\gamma_\infty}(x)
\\
&\simeq 
e^{(p-1) R(x_j)}
\, 
 \int_{6B_j}
 \Big| 
{\,\mathcal O^2_{ \{t_i\}, N} (\psi_t * (f r_j e^{-R(\cdot)}))(x)
}\Big|^p \,dx.
 \end{align*}
From  \cite[Theorem 2.7]{Liu}  we conclude that
\begin{align*} 
\norm{\, \mathcal O^2_{ \{t_i\}, N} (\widetilde{r}_j M_t(f r_j)) \,}^p_{L^p( \gamma_\infty)}
\lesssim
e^{(p-1) R(x_j)}
\,  \int_{\mathbb{R}^n} |f(x) r_j(x)|^p\, e^{-pR(x)}\, dx
\simeq
\| f\,r_j \|_{L^p(\gamma_\infty)}^p,
 \end{align*}
uniformly in $ \{t_i\}$, $N$  and $j$.
This completes the proof of  \eqref{stima_pp_Mt} and that of the proposition.
 \end{proof}

 \section{Failure of oscillation and jumps for $1\le\varrho<2$}\label{s:Qian}

\begin{proposition}\label{prop:osc_failure_rho}
    Let $1 \le \varrho < 2$. The uniform $\varrho$-oscillation estimate in $L^p(\gamma_\infty)$ with $1\le p< \infty$ fails for the semigroup  $(\mathcal H_t)$. 
   More specifically, for each large integer $N$ we can construct a function $f_N \in L^p(\gamma_\infty)$ and a decreasing sequence $(t_i)_0^N$ in $(0,1]$ such  that 
   $$ \frac{\|\mathcal{O}^\varrho_{\{t_i\},N} (\mathcal H_t\ f_N):\,t\in (0,1]\|_{L^p(\gamma_\infty)}}
   {\|f_N\|_{L^p(\gamma_\infty)}} \to \infty
  $$
   as $N \to \infty$.
  The weak type $(1,1)$ inequality for $\varrho$-oscillation also fails.
\end{proposition}

 Clearly, it does not matter whether the sequence $(t_i)_0^N$  is increasing or decreasing.

\begin{proof}
Let $Q_1$ be the unit cube $[0,1]^n$ in $\R^n$, and $\chi_1$ its characteristic function.
We will use the 
 Rademacher functions in $[0,1]$, defined by 
\begin{equation*}
  r_k = \sum_{j=1}^{2^{k-1}} \left( \chi_{[(2j-2)2^{-k },\;(2j-1)2^{-k })} -  \chi_{[(2j-1)2^{-k },\; 2j2^{-k })} \right), \qquad k=1,2,\dots.
\end{equation*}
Set 
$$
g_k(u) = \chi_1(u)\, r_k(u_1),  \quad u \in \R^n.
$$
Fix a large integer $N$. We  will estimate the oscillation of $\mathcal H_t f_N$, where     
$$
f_N(u) = \sum_{\ell = 1}^N g_{k_\ell}(u) ,  \quad u \in \R^n.
$$
Here $(k_\ell)_1^N$ is an increasing sequence to be defined recursively. We also associate to each frequency $k_\ell$ a suitable time $t_\ell \in (0,1]$, such that $t_1 > t_2 > \dots > t_N$.

Our proof relies on the following two properties of the semigroup $\mathcal H_t$.
\begin{enumerate}
    \item Let $f \in L^2(\gamma_\infty)$. 
    Then $\|\mathcal H_t f - f\|_{L^2 (\gamma_\infty)}\to 0$ as $t \to 0^+$. In particular,  $(\mathcal H_t f - f)\,\chi_1$ tends to 0 in $ L^2(dx)$
                       \item Let $t>0$ be fixed. Then $\|\chi_1\,\mathcal H_t g_k\|_{L^2(dx)} \to 0$ as  $k \to \infty$.
      \end{enumerate}

Here item (1) is clear, since the semigroup is strongly continuous. Item (2) follows from the sign changes of $r_k$ at the points $m2^{-k}, \; m= 1,\dots, 2^{-k}-1$ that imply for $x \in Q_1$
$$
|\mathcal H_t r_k(x)| \le 2^{-k} \sup_{x,u \in Q_1} |\partial_{u_1} K_t(x,u)|,
$$
and the supremum here is finite.

We denote by $\delta$  a positive number determined later. By $L^2$ we will mean $L^2(Q_1;dx)$ 
in the rest of this proof. Since $f_N$ is supported in $Q_1$, and 
we will consider the values of $\mathcal H_t f_N$ only in  $Q_1$, we can use Lebesgue measure instead of $d\gamma_\infty$.

The recursion is started with $k_0 = t_0 = 1$ (though $k_0$ is not used).

Assuming now that $k_0, t_0, \dots, k_{\ell-1}, t_{\ell-1}$ have already been chosen for some $\ell \ge 1$, we proceed as follows.

First, by item (2), we can choose $k_\ell > k_{\ell-1}$ so large that 
    \[\max_{0 \le i \le \ell-1} \|\mathcal H_{t_i} g_{k_\ell}\|_{L^2} < \delta.\] 
    This is possible because the previous times $t_1, \dots, t_{\ell-1}$ are fixed.

 Next, by item (1), we choose $t_\ell \in (0, t_{\ell-1})$ so small that 
 $$
    \max_{1 \le j \le \ell} \|\mathcal H_{t_\ell} g_{k_j} - g_{k_j}\|_{L^2} < \delta. 
    $$
    This is possible because the frequencies $k_1, \dots, k_\ell$ are now fixed.
    
 The recursion stops with $k_N$ and $ t_N$.  
 Then  the following two properties will be satisfied
for all $ \ell \in \{1, \dots, N\}$   and all $i \in \{0, 1, \dots, N\}$:

\begin{eqnarray}   
  (a) \quad     \text{if } \ell > i, \quad & \|\mathcal H_{t_i} g_{k_\ell}\|_{L^2} < \delta. \label{ett} \\
      (b) \quad      \text{if } \ell \le i, \quad & \|\mathcal H_{t_i} g_{k_\ell} - g_{k_\ell}\|_{L^2} < \delta. \label{tva}
\end{eqnarray}

We consider for each $i = 1,\dots,N$ and $x \in Q_1$ the differences
\begin{equation} % ment_decomposition}
   \mathcal H_{t_i} f_N(x) - \mathcal H_{t_{i-1}} f_N(x) =\sum_{\ell = 1}^N (\mathcal H_{t_i} g_{k_\ell}(x) - \mathcal H_{t_{i-1}} g_{k_\ell}(x))= I +II + III.
\end{equation}
where
$$
I =\sum_{\ell = 1}^{i-1} (\mathcal H_{t_i} g_{k_\ell} - \mathcal H_{t_{i-1}} g_{k_\ell}),
$$
$$
II = \mathcal H_{t_i} g_{k_i} 
$$
and
$$
III =  -  \mathcal H_{t_{i-1}} g_{k_i} + \sum_{\ell = i+1}^N \mathcal H_{t_i} g_{k_\ell} -   \sum_{\ell = i+1}^N  \mathcal H_{t_{i-1}} g_{k_\ell}.
$$
Here we omit the argument $x$.
Observe that the sum $I$ vanishes for $i=1$, and that the time $t_0$ appears only in  $III$ when    $i=1$.

To deal with $I$, we rewrite it as 
$$ 
I =\sum_{\ell = 1}^{i-1} (\mathcal H_{t_i} g_{k_\ell}- g_{k_\ell}) -  \sum_{\ell = 1}^{i-1}  (\mathcal H_{t_{i-1}} g_{k_\ell}- g_{k_\ell}).
$$
Since  $(b)$        
applies to each term here, it follows that $\|I\|_{L^2}  \le 2(i-1)\delta \le 2N\delta$.

As for the sum $III$, notice that $(a)$        
is applicable to each of the terms, and thus $\|III\|_{L^2}   < (2N+1)\delta$.

Finally  $(b)$                                  
implies that $\|II -g_{k_i}\|_{L^2}  < \delta$.

Letting
\begin{equation}  
    \mathcal{E}_i(x) = \mathcal H_{t_i} f_N(x) - \mathcal H_{t_{i-1}} f_N(x) -g_{k_i}(x),
       \end{equation}
we can sum this up as
\begin{equation} 
    \|\mathcal{E}_i\|_{L^2} < (4N+2)\delta. 
\end{equation}

Choosing now  $\delta = 1/(2(4N+2)\sqrt{2N})$,
we obtain from Chebyshev's inequality in $L^2$ that for each $i$
\begin{equation} 
    \left|\left\{x \in Q_1: |\mathcal{E}_i(x)| > \frac12 \right\}\right| \le \frac{(4N+2)^2\delta^2}{1/4} = \frac1{2N}.
\end{equation}
Considering the complementary sets, we conclude that the set
\begin{equation}
  S =  \left\{x \in Q_1: |\mathcal{E}_i(x)| \le \frac12, \;\;\;i = 1,\dots, N \right\}\
\end{equation}
has measure at least 1/2.  For $x \in S$ one has from the triangle inequality
\begin{equation} 
   |\mathcal H_{t_i} f_N(x) - \mathcal H_{t_{i-1}} f_N(x)| \ge |g_{k_i}(x)| - |\mathcal{E}_i(x)| \ge 1/2,   \qquad i = 1,\dots, N,
   \end{equation}
   the last inequality since $|g_{k_i}| = 1$ in $Q_1$. 
   
   For the $\varrho$-oscillation of $\mathcal H_{t} f_N$ with the sequence $(t_i)_0^N$ we then get 
\begin{equation} 
\mathcal O^\varrho_{ \{t_i\}, N}  ( \mathcal H_t (x) f_N:\,t\in (0,1]) \ge \left( \sum_{1}^{N} |\mathcal H_{t_i} f_N(x) - \mathcal H_{t_{i-1}} f_N(x)|^\varrho\right)^{1/\varrho}  \ge \frac12\,N^{1/\varrho},
 \end{equation}
 when  $x \in S$. Now take the $L^p$ norm in $Q_1$ with $1 \le p<\infty$, getting 
\begin{equation} 
\|\mathcal O^\varrho_{ \{t_i\}, N}  ( \mathcal H_t (x) f_N:\,t\in (0,1])\|_{L^p} \gtrsim N^{1/\varrho}.
 \end{equation}
For $\varrho < 2$  and large $N$ this is much larger than the $L^p$ norm of $f_N$, since Khinchine's inequality shows that $\|f_N\|_{L^p} \simeq N^{1/2}$.
Thus there can be no $L^p$ oscillation inequality for $\mathcal H_t$ with $\varrho < 2$. 
A weak type $(1,1)$ inequality is also impossible, since the oscillation is larger than $N^{1/\varrho}/2$ on a set of measure at least $1/2$.
\end{proof}

\begin{remark}\label{remark:jumps_failure_Ht}
As pointed out in items (3') and (4') in Subsection \ref{subs:jump},
 the uniform $\varrho$-jump bounds fail for the full Ornstein--Uhlenbeck semigroup 
 $(\mathcal{H}_t)$ when $1 \le \varrho < 2$. This failure can be proved via an abstract argument originally due to Bourgain (see also \eqref{ineq:Mariusz}). 

However, using the  family of functions $(f_N)$
constructed for the oscillation in the proof of Proposition \ref{prop:osc_failure_rho} above, we can show this failure in a much more direct way. On the  set $S$ used above, the sequence $(\mathcal{H}_{t_i} f_N(x))_{i=1}^N$ exhibits $N$ consecutive jumps of size at least $1/2$. Consequently, for the threshold $\lambda = 1/4$, the jump counting function satisfies $N_{1/4}(\mathcal{H}_t f_N(x) : t \in (0,1]) = N$. It follows that
\begin{equation}
    J_\varrho(\mathcal{H}_t f_N : t \in (0,1])(x) = \sup_{\lambda > 0} \lambda \big(
     N_\lambda(\mathcal{H}_t f_N(x) : t \in (0,1]) \big)^{1/\varrho} \ge \frac{1}{4} N^{1/\varrho}.
\end{equation}
Taking the $L^p$ norm, we have
\[ 
    \frac{ \|J_\varrho (\mathcal{H}_t f_N : t \in (0,1])\|_{L^p}}{\|f_N\|_{L^p}} 
 \gtrsim N^{\frac{1}{\varrho} - \frac{1}{2}}, 
\qquad 1 \le p < \infty, 
\]
which diverges as $N \to \infty$ since $\varrho < 2$.
\end{remark}

\section{Moving from Gaussian to Lebesgue measure}\label{s:Leb}
We conclude by discussing  briefly what occurs when we replace the invariant measure 
in $\R^n$
with the standard 
 Lebesgue measure.
Since the semigroup $(\mathcal H_t)_{t>0}$ is not analytic in
$L^p(dx) := L^p(\R^n, dx)$, a  large part of the tools 
our proofs rely on  
are no longer available  ({\em in primis,}
Corollary 4.5 in \cite{Le Merdy}, 
which guarantees the
$L^p(\gamma_\infty)$
boundedness 
of the $\varrho$-th order variation   of $(\mathcal H_t)_{t>0}$ 
 for $\varrho>2$).
 
 Indeed, the following result holds.

\begin{proposition}
  For  $1 \le \varrho < \infty$ and
 $1 \le p < \infty$, 
   the maximal, variational, and jump operators
   and 
the oscillation associated with the Ornstein--Uhlenbeck semigroup 
are neither strongly nor weakly bounded on
 $L^p(dx) $.
\end{proposition}
 Notice, however,  
 that   the local parts of these  operators in the Euclidean setting  behave exactly like their Gaussian counterparts,
due to the local equivalence between $dx$ and $d\gamma_\infty$.

\medskip
\begin{proof}
 
Our argument goes via the maximal operator $\mathcal H_*f=\sup_{t>0} \mathcal H_t f$.

Consider a non-negative function $f \in L^1(\gamma_\infty) \cap L^1(dx)$, not identically zero, with compact support.
 For any $x \in \R^n$,
 Kolmogorov's formula \eqref{Kolmo} 
 yields
 \[
\lim_{t\to +\infty} \mathcal H_t f(x) = \int_{\R^n}f(y)\,d\gamma_\infty(y) =: \ell_\infty > 0.
\]
Then we take $0<\alpha<\ell_\infty$ and see that the level set of $\mathcal H_*f$ is
$$
\{ x\in\R^n: \mathcal H_* f(x)>\alpha \} = \R^n,
$$
of infinite Lebesgue measure. Thus $\mathcal H_*$ satisfies neither weak $(1,1)$ nor strong $(p,p)$ estimates
(but $\mathcal H_*$ is bounded from $L^\infty (dx)$ to $L^\infty(dx)$, as seen from  Kolmogorov's formula \eqref{Kolmo}).
\medskip

As a consequence, the variational seminorms of $\|\mathcal H_tf(x)\|_{v(\varrho),\R_+}$ with any $\varrho \ge 1$  are unbounded on $L^p(dx)$ for $1\le p<\infty$, also in a weak sense.
 Indeed, for any
 $\varrho \ge 1$    
 the simple pointwise bound 
\[\mathcal{H}_* f(x) \le |f(x)| + \|\mathcal{H}_t f(x)\|_{v(\varrho), \R_+}\] holds. 
Since $f \in L^p(dx)$, if the variation operator were bounded in $L^p(dx)$,
 this inequality would imply the $L^p(dx)$ boundedness of $\mathcal{H}_* $, which is false. 

\medskip

The same failure applies to the jump operator and to oscillation
 with $\varrho \ge 1$, as we will now see.

We choose $f$ as above and notice that
for any $x \notin \text{supp}f$, we have $\lim_{t\to 0} \mathcal{H}_t f(x) = f(x) = 0$. 
The function $t \mapsto \mathcal{H}_t f(x)$ is continuous and converges to $\ell_\infty>0$ as $t \to +\infty$.
 We consider a two-point sequence $(t_i)_0^1$ and see that for     $x \notin \text{supp}f$ 
 \begin{equation}\label{counterex}
 |\mathcal H_{t_1} f(x)  - \mathcal H_{t_0} f(x)|    \to \ell_\infty  \qquad \text{as} \qquad t_0 \to 0, \;\; t_1 \to \infty.
 \end{equation}
 For each  $x \notin \text{supp}f$ 
 we can thus find a sequence  $(t_i)_0^1$ such that $|\mathcal H_{t_1} f(x)  - \mathcal H_{t_0} f(x)| > \ell_\infty/2$,
and so the jump counting function satisfies $N_{\ell_\infty/2} (\mathcal{H}_t f(x):\,t > 0) \ge 1$ 
on a set of infinite Lebesgue measure. 
This excludes strong and weak  $ L^p(dx)$ boundedness for the jump operator.

To deal with   the oscillation, we observe that the convergence of $\mathcal{H}_t f(x)$ as $t \to 0$ or $t \to \infty$ is locally uniform in 
 $\R^n \setminus \text{supp}f$, and so is that in \eqref{counterex}. Since
 $$
 \mathcal O^\varrho_{\{t_i\}, 1} (\mathcal H_t f(x):\,t> 0) \ge  |\mathcal H_{t_1} f(x)  - \mathcal H_{t_0} f(x)|,
$$
this uniform  convergence means that we can select one sequence  $(t_i)_0^1$ which makes $ \mathcal O^\varrho_{\{t_i\}, 1} (\mathcal H_t f(x):\,t> 0) > \ell_\infty/2$ 
on a set of arbitrarily large Lebesgue measure. Hence, no strong or weak uniform oscillation inequalities can hold.

\end{proof}

 \end{document}